\documentclass[12pt]{amsart}       
\usepackage{txfonts}
\usepackage{amssymb}
\usepackage{eucal}
\usepackage{graphicx}
\usepackage{amsmath}
\allowdisplaybreaks[4]
\usepackage{mathrsfs}
\usepackage{amscd}
\usepackage[all]{xy}           
\usepackage{amsfonts,latexsym}
\usepackage{xspace}
\usepackage{epsfig}
\usepackage{float}
\usepackage{color}
\usepackage{fancybox}
\usepackage{multicol}
\usepackage{colordvi}
\usepackage{ifpdf}
\ifpdf
\usepackage[colorlinks,final,backref=page,hyperindex]{hyperref}
\else
\usepackage[colorlinks,final,backref=page,hyperindex,hypertex]{hyperref}
\fi
\usepackage[active]{srcltx} 

\usepackage{tikz}
\usepackage{graphicx}
\usepackage{bbm}

\newtheorem{theorem}{Theorem}[section]
\newtheorem{prop}[theorem]{Proposition}
\newtheorem{lemma}[theorem]{Lemma}
\newtheorem{coro}[theorem]{Corollary}
\newtheorem{prop-def}{Proposition-Definition}[section]
\newtheorem{coro-def}{Corollary-Definition}[section]
\theoremstyle{definition}
\newtheorem{defn}[theorem]{Definition}

\newtheorem{remark}[theorem]{Remark}

\newtheorem{exam}[theorem]{Example}

\newcommand{\nc}{\newcommand}
\nc{\tred}[1]{\textcolor{red}{#1}}
\nc{\tblue}[1]{\textcolor{blue}{#1}}
\nc{\tgreen}[1]{\textcolor{green}{#1}}
\nc{\tpurple}[1]{\textcolor{purple}{#1}}
\nc{\btred}[1]{\textcolor{red}{\bf #1}}
\nc{\btblue}[1]{\textcolor{blue}{\bf #1}}
\nc{\btgreen}[1]{\textcolor{green}{\bf #1}}
\nc{\btpurple}[1]{\textcolor{purple}{\bf #1}}

\nc{\name}[1]{{\bf #1}}

\nc{\NN}{{\mathbb N}}
\nc{\ncsha}{{\mbox{\cyr X}^{\mathrm NC}}} \nc{\ncshao}{{\mbox{\cyr
X}^{\mathrm NC}_0}}

\newcommand{\efootnote}[1]{}

\renewcommand{\textbf}[1]{}

\newcommand{\delete}[1]{}

\nc{\mlabel}[1]{\label{#1}}  
\nc{\mcite}[1]{\cite{#1}}  
\nc{\mref}[1]{\ref{#1}}  
\nc{\meqref}[1]{\eqref{#1}}  
\nc{\mbibitem}[1]{\bibitem{#1}} 

\delete{
\nc{\mlabel}[1]{\label{#1}  
{\hfill \hspace{1cm}{\small\tt{{\ }\hfill(#1)}}}}
\nc{\mcite}[1]{\cite{#1}{\small{\tt{{\ }(#1)}}}}  
\nc{\mref}[1]{\ref{#1}{{\tt{{\ }(#1)}}}}  
\nc{\meqref}[1]{\eqref{#1}{{\tt{{\ }(#1)}}}}  
\nc{\mbibitem}[1]{\bibitem[\bf #1]{#1}} 
}

\nc{\bfx}{\mathbf{x}}
\nc{\opa}{\ast} \nc{\opb}{\odot} \nc{\op}{\bullet} \nc{\pa}{\frakL}
\nc{\arr}{\rightarrow} \nc{\lu}[1]{(#1)} \nc{\mult}{\mrm{mult}}
\nc{\diff}{\mathfrak{Diff}}
\nc{\opc}{\sharp}\nc{\opd}{\natural}
\nc{\ope}{\circ}
\nc{\dpt}{\mathrm{d}}
\nc{\hck}{H_{RT}}
\nc{\vdf}{\calf}
\nc{\ldf}{\calf_\ell}
\nc{\hlf}{H_\ell}
\nc{\onek}{\mathbf{1}_\bfk}

\nc{\ot}{\otimes}
\nc{\oot}{\otimes}
\nc{\mot}{{{\boxtimes\,}}}
\nc{\otm}{\overline{\boxtimes}} \nc{\sprod}{\bullet}
\nc{\scs}[1]{\scriptstyle{#1}} \nc{\mrm}[1]{{\rm #1}}
\nc{\margin}[1]{\marginpar{\rm #1}}   
\nc{\dirlim}{\displaystyle{\lim_{\longrightarrow}}\,}
\nc{\invlim}{\displaystyle{\lim_{\longleftarrow}}\,}
\nc{\mvp}{\vspace{0.3cm}} \nc{\tk}{^{(k)}} \nc{\tp}{^\prime}
\nc{\ttp}{^{\prime\prime}} \nc{\svp}{\vspace{2cm}}
\nc{\vp}{\vspace{8cm}} \nc{\proofbegin}{\noindent{\bf Proof: }}
\nc{\proofend}{$\blacksquare$ \vspace{0.3cm}}
\nc{\modg}[1]{\!<\!\!{#1}\!\!>}
\nc{\intg}[1]{F_C(#1)} \nc{\lmodg}{\!
<\!\!} \nc{\rmodg}{\!\!>\!}
\nc{\cpi}{\widehat{\Pi}}
\nc{\sha}{{\mbox{\cyr X}}}  
\nc{\shap}{{\mbox{\cyrs X}}} 
\nc{\shpr}{\diamond}    
\nc{\shp}{\ast} \nc{\shplus}{\shpr^+}
\nc{\shprc}{\shpr_c}    
\nc{\msh}{\ast} \nc{\zprod}{m_0} \nc{\oprod}{m_1}
\nc{\vep}{\epsilon} \nc{\labs}{\mid\!} \nc{\rabs}{\!\mid}
\nc{\sqmon}[1]{\langle #1\rangle}
\nc{\wmn}{\widetilde{\mathbb{N}}}
\nc{\R}{R}

\nc{\bfk}{{\bf k}}
\nc{\PP}{{\mathbb P}}
\nc{\QQ}{{\mathbb Q}}\nc{\RR}{{\mathbb R}} \nc{\ZZ}{{\mathbb Z}}

\nc{\cala}{{\mathcal A}} \nc{\calb}{{\mathcal B}}
\nc{\calc}{{\mathcal C}}
\nc{\cald}{{\mathcal D}} \nc{\cale}{{\mathcal E}}
\nc{\calf}{{\mathcal F}} \nc{\calg}{{\mathcal G}}
\nc{\calh}{{\mathcal H}} \nc{\cali}{{\mathcal I}}
\nc{\call}{{\mathcal L}} \nc{\calm}{{\mathcal M}}
\nc{\caln}{{\mathcal N}} \nc{\calo}{{\mathcal O}}
\nc{\calp}{{\mathcal P}} \nc{\calr}{{\mathcal R}}
\nc{\cals}{{\mathcal S}} \nc{\calt}{{\mathcal T}}
\nc{\calu}{{\mathcal U}} \nc{\calw}{{\mathcal W}} \nc{\calk}{{\mathcal K}}
\nc{\calx}{{\mathcal X}} \nc{\CA}{\mathcal{A}}

\nc{\fraka}{{\mathfrak a}} \nc{\frakA}{{\mathfrak A}}
\nc{\frakb}{{\mathfrak b}} \nc{\frakB}{{\mathfrak B}}
\nc{\frakD}{{\mathfrak D}} \nc{\frakF}{\mathfrak{F}}
\nc{\frakf}{{\mathfrak f}} \nc{\frakg}{{\mathfrak g}}
\nc{\frakH}{{\mathfrak H}} \nc{\frakL}{{\mathfrak L}}
\nc{\frakM}{{\mathfrak M}} \nc{\bfrakM}{\overline{\frakM}}
\nc{\frakm}{{\mathfrak m}} \nc{\frakP}{{\mathfrak P}}
\nc{\frakN}{{\mathfrak N}} \nc{\frakp}{{\mathfrak p}}
\nc{\frakS}{{\mathfrak S}} \nc{\frakT}{\mathfrak{T}}
\nc{\frakX}{{\mathfrak X}}
\nc{\BS}{\mathbb{S
}}

\font\cyr=wncyr10 \font\cyrs=wncyr7
\nc{\li}[1]{\textcolor{red}{#1}}
\nc{\lir}[1]{\textcolor{red}{Li:#1}}
\nc{\xing}[1]{\textcolor{blue}{Xing:#1}}
\nc{\song}[1]{\textcolor{blue}{Song:#1}}
\nc{\revise}[1]{\textcolor{red}{#1}}
\nc{\lid}[1]{{}}
\nc{\xingd}[1]{{}}
\nc{\songd}[1]{{}}

\nc{\id}{\mathrm{id}}
\nc{\wn}{\widetilde{\mathbb{N}}}
\nc{\wip}{\widetilde{\mathbb{P}}}
\nc{\var}{\varepsilon}
\nc{\pvar}{\mathcal{P}^{\varepsilon}}
\nc{\avar}{\mathcal{A}^{\varepsilon}}
\nc{\gvar}{\Gamma^{\varepsilon}}
\nc{\weight}{\mathrm{weight}}

\nc{\etree}{\mathbbm{1}}
\nc{\nx}{\widetilde{\mathbb{N}}^{A}}

\nc{\wmp}{\widetilde{\mathbb{P}}}
\nc{\E}{E}

\nc{\e}{a}

\nc{\sbx}{\mathcal{MC}(\E)}
\nc{\snx}{\mathcal{MC}(\widetilde{\mathbb{N}})}
\nc{\spx}{\mathcal{MC}(\widetilde{\mathbb{P}})}
\nc{\mcwse}{\text{multi-composition with semigroup exponents}}
\nc{\mcswse}{\text{multi-compositions with semigroup exponents}}
\nc{\qsym}{\mathrm{MQSym}^{\E}}
\nc{\qsymp} {\mathrm{MQSym}^{\widetilde{\mathbb{P}}}}
\nc{\mqsymn} {\mathrm{MQSym}^{\widetilde{\mathbb{N}}}}
\nc{\wmqsym}{\text{WMQSym}\xspace}

\nc{\nomega}{\widetilde{\mathbb{N}}^{\Omega}}
\nc{\saunx}{\mathrm{sau(\widetilde{\mathbb{N}}^{\Omega})}}
\nc{\sausnx}{\mathrm{sau(\mathcal{MC}(\widetilde{\mathbb{N}}))}}
\nc{\Set}{\mathrm{Set}}

\nc{\wsnx}{\mathcal{WMC}(\mathbb{N})}

\nc{\X}{X}\nc{\x}{x}\nc{\rOmega}{\Omega}
\nc{\mqsym}{\mathrm{MQSym}}

\nc{\qsa}{\mathrm{QS}(A)}
\nc{\sm}{[m]}
\nc{\A}{[m]}
\nc{\ea}{\A^{\E}}
\nc{\ax}{i}
\nc{\wnx}{[m]^{\mathbb{N}}}
\nc{\sqsymnx}{\mathrm{SMQSym}^{\widetilde{\mathbb{N}}}(\sm)}
\nc{\noi}{\text{command}}
\nc{\mcomp}{\mathrm{MComp}}
\nc{\mcompn}{{\rm MComp}(n)}
\nc{\opar}{\lhd}
\nc{\opare}{\unlhd}
\nc{\Des}{\mathrm{SN}}

\begin{document}
\title[Multi-quasisymmetric functions and Rota-Baxter algebras]{Multi-quasisymmetric functions with semigroup exponents, Hopf algebras and Rota-Baxter algebras}
%
\author{Xing Gao}
\address{School of Mathematics and Statistics, Lanzhou University,
Lanzhou, 730000, China; Gansu Provincial Research Center for Basic Disciplines of Mathematics and Statistics, Lanzhou, 730070, China;
School of Mathematics and Statistics
Qinghai Nationalities University, Xining, 810007, China}
         \email{gaoxing@lzu.edu.cn}

\author{Li Guo}
\address{Department of Mathematics and Computer Science, Rutgers University at Newark, Newark, New Jersey, 07102, United States}

\author{Xiao-Song Peng}
\address{School of Mathematics and Statistics,
Jiangsu Normal University, Xuzhou, Jiangsu 221116, China}
\email{pengxiaosong@jsnu.edu.cn}

\date{\today}
\begin{abstract}
Many years ago, G.-C.~Rota discovered a close connection between symmetric functions and Rota-Baxter algebras, and proposed to study generalizations of symmetric functions in the framework of Rota-Baxter algebras. Guided by this proposal, quasisymmetric functions from weak composition (instead of just compositions) were obtained from free Rota-Baxter algebras on one generator. This paper aims to generalize this approach to free Rota-Baxter algebras on multiple generators in order to obtain further generalizations of quasisymmetric functions.
For this purpose and also for its independent interest, the space $\mathrm{MQSym}$ of quasisymmetric functions on multiple sequences of variables is defined, generalizing quasisymmetric functions and diagonally quasisymmetric functions of Aval, Bergeron and Bergeron. Linear bases of such multi-quasisymmetric functions are given by monomial multi-quasisymmetric functions and fundamental multi-quasisymmetric functions, the latter recover the fundamental $G^m$-quasisymmetric functions of Aval and Chapoton.
Next introduced is the even more general notion of multi-quasisymmetric functions $\qsym$ with exponents in a semigroup $E$, which also generalizes the quasisymmetric functions with semigroup exponents in a recent work.
Through this approach, a natural Hopf algebraic structure is obtained on $\qsym$.
Finally, in support of Rota's proposal, the free commutative unitary Rota-Baxter algebra on a finite set is shown to be isomorphic to a scalar extension of a subalgebra of $\mqsymn$, a fact which in turn equips the free Rota-Baxter algebra with a Hopf algebra structure.
\end{abstract}

\makeatletter
\@namedef{subjclassname@2020}{\textup{2020} Mathematics Subject Classification}
\makeatother
\subjclass[2020]{
05E05, 
16W99, 
16S10, 
17B38, 
08B20, 
16T30,  
}

\keywords{quasisymmetric function, symmetric function, composition, Rota-Baxter algebra, Hopf algebra, quasi-shuffle}

\maketitle

\vspace{-1cm}

\tableofcontents

\setcounter{section}{0}

\section{Introduction}

This paper studies the multiple generation of quasisymmetric functions, both in the classical case of compositions and the recent generalization of weak compositions, together with their relationship with Rota-Baxter algebras, motivated by a program proposed by Rota.

\vspace{-.2cm}

\subsection{Symmetric and quasisymmetric functions, Rota-Baxter algebras and Rota's proposal}
We first recall symmetric functions, and quasisymmetric functions from compositions and weak compositions, together with their relationship with the free Rota-Baxter algebra on one generator.

The notion of symmetric functions is central in algebraic combinatorics and has broad applications to group theory, Lie algebra and algebraic geometry~\mcite{Ma95,St99}. With the Hopf algebraic structure, its study motivated many related structures since the 1990s.

One such generalization is quasisymmetric functions, defined by Gessel \mcite{Ges84} in 1984 with motivation from the P-partitions proposed by Stanley~\mcite{Sta72}.
The study of quasisymmetric functions has interacted with broad areas in mathematics including Hopf algebras~\mcite{Ehr90, MR95}, discrete geometry~\mcite{BHW03}, representation theory~\mcite{Hiv00} and algebraic topology~\mcite{BR08}.

On the other hand, the study of Rota-Baxter algebras originated from the 1960 work~\mcite{Ba} of G. Baxter in fluctuation theory of probability. During the 1960-70s, well-known mathematicians such as Atkinson, Cartier and Rota~\mcite{Atk63,Ca,Ro} studied Rota-Baxter algebras from analytic and combinatorial viewpoints. In the recent decades, Rota-Baxter algebras have found remarkable applications in diverse areas in mathematics and mathematical physics,
most notably the work of Connes and Kreimer on renormalization of quantum field theory~\mcite{CK}. See~\mcite{G12} and the references therein.

The connection between symmetric functions and Rota-Baxter algebras was first established by Rota in~\mcite{Ro} where he showed that Waring's identity relating power sum symmetric functions and elementary symmetric functions can be formulated as a special case of the Spitzer's identity for the free commutative Rota-Baxter algebra.
Then in his inspiring survey article~\mcite{Rot95}, Rota proposed that there is a close relationship between Rota-Baxter algebras and generalizations of symmetric functions.
He conjectured that
\begin{quote}
	a very close relationship exists between the (Rota-)Baxter identity and the algebra of symmetric functions.
\end{quote}
and proposed
\begin{quote}
(Rota-)Baxter algebras represent the ultimate and most natural generalization of the algebra of symmetric functions.
\end{quote}

As an evidence in support of this proposal, the equivalence of the mixable shuffle product~\mcite{GK00} in a free commutative Rota-Baxter algebra and the quasi-shuffle product~\mcite{H00} generalizing quasisymmetric functions was obtained~\mcite{EG06}. In particular, quasisymetric functions form a canonical subalgebra of the free commutative Rota-Baxter algebra on one generator. This gives a precise testing ground of Rota's proposal, in interpreting the full commutative Rota-Baxter algebra as suitable generalizations of quasisymmetric functions.   Pursuing this direction, the concept of left weak composition (LWC) quasisymmetric functions was introduced~\mcite{YGZ17} which, after a scalar extension, is identified with the free commutative nonunitary Rota-Baxter algebra on one generator.
Furthermore, the algebra of weak quasisymmetric functions was introduced~\mcite{YGT19} which interpreted the free commutative unitary Rota-Baxter algebra on one generator as a further generalization of quasisymmetric functions. Moreover, to cure the divergency of weak quasisymmetric function as formal power series, the method of renormalization of quantum field theory is adapted in~\mcite{GYZ}, giving another construction of weak quasisymmetric functions.
\vspace{-.2cm}

\subsection{The multivariant case}
The success of applying the free commutative Rota-Baxter algebra on one generator to give generalizations of (quasi-)symmetric functions suggests to consider free commutative Rota-Baxter algebras on multiple generator.
Viewing through this lens, since symmetric functions and quasisymmetric functions are formal power series in one sequence $x_n, n\geq 0,$ of variables, one is naturally led to consider generalizations of symmetric and quasisymmetric functions in multiple sequences of variables.

As it turns out, symmetric functions in multiple sequences of variables, often called multisymmetric functions, have been extensively studied almost as long as the one sequence case, tracing back to the works of Schl\"afli, MacMahon, E. Noether and Weyl~\mcite{Mac60,No15,Sch1852,Wel39}. See~\mcite{Br2004} for a historical perspective and see~\mcite{BB20,Br2004,Co,Dal99,DP12,GRW,HPR18,RS,Va05,Va2} for some of the recent developments and applications.

In comparison, the study of multiple generalizations of quasisymmetric functions is more recent. The case of double sequence variables is studied under the name of \name{diagonal quasisymmetric functions}~\mcite{ABB06,Po98}.
For the multiple-sequence case, \name{fundamental $G^m$-quasisymmetric functions} was introduced in~\mcite{AC16} and denoted ${\rm QSym}(G^m)$, for $m$ copies a variable set $G$.
In~\cite{BH08}, the case of two copies of variable is realized as the graded dual Hopf algebra of the analog in type $B$ of Solomon's descent algebra.

As noted above, from the viewpoint of Rota, the notion of multi-quasisymmetric functions is naturally expected since the classical case corresponds to the case of the free commutative Rota-Baxter algebra on one generator.
This is the approach that we take in this paper. For their independent interest, quasisymmetric functions on multiple sequences of variables, called \name{multi-quasisymmetric functions}, are defined motivated by the diagonal quasisymmetric functions. They are shown to have linear bases given by monomial multi-quasisymmetric functions and by fundamental multi-quasisymmetric functions, both indexed by multi-compositions. These linear bases recovers the above notions of bimonomial quasisymmetric functions when $m=2$ and fundamental $G^m$-quasisymmetric functions.

To make contact with Rota-Baxter algebras, this notion is further generalized to functions (power series) from weak composition, as shown in the following diagram where all the arrows are inclusions.
\smallskip

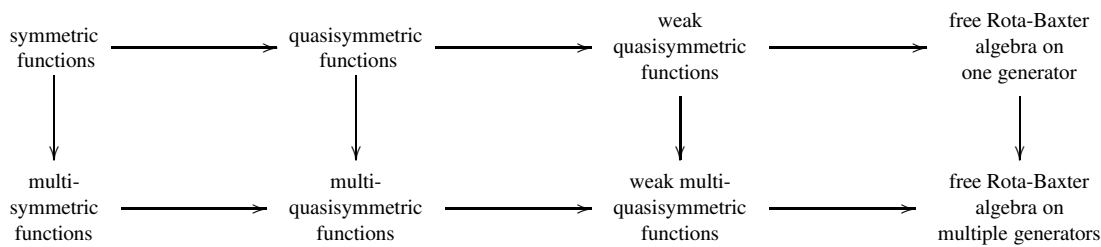
\xymatrix{
\text{symmetric} \atop \text{ functions} \ar[rr] \ar[d] && \text{quasisymmetric}\atop \text{ functions} \ar[rr] \ar[d]&& {\tiny \begin{array}{c} \text{weak} \\ \text{quasisymmetric}\\ \text{functions}\end{array}}\ar[rr]\ar[d] && {\tiny \begin{array}{c} \text{free Rota-Baxter} \\ \text{algebra on}\\ \text{one generator}\end{array}}\ar[d] \\
{\tiny \begin{array}{c} \text{multi-} \\ \text{symmetric}\\ \text{functions}\end{array}} \ar[rr] && {\tiny \begin{array}{c} \text{multi-} \\ \text{quasisymmetric}\\ \text{functions}\end{array}} \ar[rr] && {\tiny \begin{array}{c} \text{weak multi-} \\ \text{quasisymmetric}\\ \text{functions}\end{array}} \ar[rr] && {\tiny \begin{array}{c} \text{free Rota-Baxter} \\ \text{algebra on}\\ \text{multiple generators}\end{array}}
}
\medskip

By constructing weak multi-quasisymmetric functions, we are able to obtain a further generalization of multi-quasisymmetric functions and at the same time to give a power series realization of the abstractly defined free commutative Rota-Baxter algebra on multiple generators, thereby giving further evidence in support of Rota's proposal.

\vspace{-.3cm}

\subsection{Layout of the paper}

The plan of the paper is as follows.

In Section~\mref{seca}, we first recall the needed background on quasisymmetric functions. Generalizing the concepts of quasisymmetric functions and diagonally quasisymmetric functions~\mcite{ABB06}, we introduce the set of multi-quasisymmetric functions $\mqsym$ and show that it is a vector space spanned by monomial multi-quasisymmetric functions. Fundamental multi-quasisymmetric functions are defined, mirroring the $G^m$-quasisymmetric polynomials in~\mcite{AC16}, and are shown to give another linear basis of multi-quasisymmetric functions.
To obtain a Hopf algebra structure for multi-quasisymmetric functions, we consider the even broader notion of multi-quasisymmetric functions with semigroup exponents. This notion is more directly related to quasi-shuffle algebras generated by commutative algebras and, as a special case, gives the algebra \wmqsym of weak multi-quasisymmetric functions.

Section~\mref{secb} first gives the Hopf algebra structure on general quasi-shuffle algebras. As special cases, the Hopf algebra structures on multi-quasisymmetric functions with semigroup exponents and in particular on \wmqsym are deduced.

Section~\mref{secd} establishes the relation between multi-quasisymmetric functions and Rota-Baxter algebras. First the algebra \wmqsym of weak multi-quasisymmetric functions is identified with the plus part of the free commutative Rota-Baxter algebra. Then by an extension of scalar of the weak multi-quasisymmetric functions, we obtain the free commutative Rota-Baxter algebra itself and, at the same time, equip the latter with a Hopf algebra structure.
This can be regarded as another result in support of Rota's proposal to find generalizations of symmetric functions by applying Rota-Baxter algebras.

\smallskip

{\bf Notation.}
Throughout this paper, let $\bfk$ be a unitary commutative ring containing $\mathbb{Q}$ unless the contrary is specified. It will be the base ring of all modules, algebras, tensor products, as well as linear maps.
By an algebra we mean a unitary associative algebra unless otherwise stated.
Denote by $\mathbb{N}$ the set of non-negative integers and $\mathbb{P}$ the set of positive ones.

\section{Multi-quasisymmetric functions of various exponents}\mlabel{seca}

Generalizing the multi-symmetric functions~\mcite{Dal99}, diagonally quasisymmetric functions~\mcite{ABB06} and the quasisymmetric functions with semigroup exponents, as well as weak quasisymmetric functions~\mcite{YGT19}, we introduce the notion of multi-quasisymmetric functions with semigroup exponents, before studying their properties and applications in later sections. Canonical linear bases of multi-quasisymmetric functions recovers the notions of bimonomial quasisymmetric functions~\mcite{ABB06} and fundamental $G^m$-quasisymmetric functions~\mcite{AC16}.

\subsection{Multi-quasisymmetric functions}
\mlabel{ss:mqsym}

A {\bf composition} is an ordered tuple of positive integers $\alpha=(\alpha_1,\ldots, \alpha_k)$. If $\sum \limits_{i=1}^k \alpha_i=n$, then $\alpha$ is called a \name{composition of $n$}, written $\alpha \models n$.
A {\bf weak composition} is an ordered tuple of nonnegative integers.

For an infinite sequence of variables $\bfx:=(x_1,x_2,\ldots,)$, a {\bf quasisymmetric function} in $\bfx$ is a formal power series $Q(\bfx)=Q(x_1,x_2,\ldots)$ in $\bfk[[\bfx]]=\bfk[[x_1,x_2,\ldots]]$ of bounded degree
such that for any composition $\alpha=(\alpha_1,\ldots, \alpha_k)$, the coefficient of $x_1^{\alpha_1}\cdots x_{k}^{\alpha_k}$ equals to the coefficient of $x_{i_1}^{\alpha_1}\cdots x_{i_k}^{\alpha_k}$ for all $ i_1< \cdots <i_k$ in $\PP$.
Denote by $\mathrm{QSym}_n$ the set of all quasisymmetric functions homogeneous of degree $n$. Then
$$\mathrm{QSym}\coloneqq \mathop{\bigoplus} \limits_{n \geq 0} \mathrm{QSym}_n$$
is the graded algebra of all quasisymmetric functions, with $\mathrm{QSym}_0=\bfk$.

A standard linear bases of $\mathrm{QSym}$ is given by the set of {\bf monomial quasisymmetric functions} $M_\alpha$ parameterized by compositions $\alpha=(\alpha_1,\ldots, \alpha_k)$:
\begin{align*}
M_{\alpha}\coloneqq\sum \limits_{i_1 <\cdots < i_k} x_{i_1}^{\alpha_1}\cdots x_{i_k}^{\alpha_k},
\end{align*}
where the sum runs over all ordered tuples ${\bf i}\coloneqq\{i_1<\cdots<i_k\}$ in $\PP$.

For two infinite sequences, Aval, Bergeron and Bergeron defined bicompositions and diagonally quasisymmetric functions in~\mcite{ABB06}. We now give the more general notions of multi-compositions and quasisymmetric functions for multiple sequences of variables.

Throughout the rest of this paper, we fix a positive integer $m$. Denote $\sm:=\{ 1,2, \ldots, m\}$. For each $i \in \sm$, let ${\bf x}_{i}:=\{ x_{i,1}, x_{i,2}, \ldots \}$ be a sequence of mutually commuting variables and define, as a disjoint union, the set of variables
\begin{equation}
\X := {\bf x}_{\sm}:=\bigsqcup_{i \in \sm} {\bf x}_{i}=\bigsqcup_{i \in \sm} \, \big\{\x_{i,1}, \x_{i,2}, \ldots \big\}.
\mlabel{eq:eks}
\end{equation}
Let $\bfk [[\X]]$ be the formal power series algebra in the variables $\X $. Consider the matrix
\begin{equation}
{\bf w}:=
\begin{pmatrix}
a_{1,1} & \ldots & a_{1,k}\\
a_{2,1} & \ldots & a_{2,k}\\
\vdots & \ddots & \vdots \\
a_{m,1} & \ldots & a_{m,k}\\
\end{pmatrix} =:(w_1,\ldots,w_k)\in \NN^{m\times k}.
\mlabel{eq:multcomp}
\end{equation}
For $w_i$, $1 \leq i \leq k$, we denote the transposed column vector $w_i^T=a_{1, i}\, a_{2, i}\, \cdots a_{m,i}$ for brevity.

\begin{defn}
A matrix ${\bf w}=(w_1, \ldots, w_k)$ in Eq.~\meqref{eq:multcomp} is called a {\bf multi-composition} or, more precisely, an {\bf $\sm$-composition of length $k$} if for each $1 \leq i \leq k$, the column vector $w_i$ is nonzero. Denote the length of ${\bf w}$ by $\ell({\bf w})$.
Let $\mcomp$ denote the set of multi-compositions.
\label{def:mc}
\end{defn}

For a sequence $w=(a_{1},a_{2},\ldots, a_{m})\in \NN^{m}$ of nonnegative integers and $j\in \PP$, define
$$x^w_j:=x_{1,j}^{a_{1}} \cdots x_{m,j}^{a_m} = \prod_{i \in \sm} x_{i,j}^{a_{i}}.$$
Then for a multi-composition ${\bf w}=(w_1, \ldots, w_k)$ of length $k$ and
${\bf j}=\{j_1 < \cdots < j_k \}\subseteq \PP$, define
\begin{align*}
x^{{\bf w}}_{{\bf j}} := x^{w_1}_{j_1} \cdots x^{w_k}_{j_k}.
\end{align*}

\begin{exam}
Consider the multi-composition of length three
\begin{align*}
{\bf w}: = (w_1,w_2,w_3):=\begin{pmatrix}
1 & 0 & 1\\
0 & 2 & 1\\
3 & 1 & 2 \\
0 & 1 & 1
\end{pmatrix}.
\end{align*}
Then
$w_1^T=1 0 3 0, \,\, w_2^T= 0 2 1 1, \,\, w_3^T=1 1 2 1$
and
\begin{align*}
x^{w_2}_{2}=x^{2}_{2, 2} x^{1}_{3,2} x^{1}_{4,2}, \,\, x^{\bf w}_{\{1<2<3 \}}=x_1^{w_1} x_2^{w_2} x_3^{w_3} = (x^{1}_{1,1}x^{3}_{3,1})(x^{2}_{2,2} x^{1}_{3,2} x^{1}_{4,2})(x^{1}_{1,3} x^{1}_{2,3} x^{2}_{3,3} x^{1}_{4,3}).
\end{align*}
\end{exam}

Generalizing the concept of quasisymmetric functions, we define multi-quasisymmetric functions as follows.
\begin{defn}
For $m\geq 1$ and $X={\bf x}_{\sm}$ as in Eq.~\meqref{eq:eks}, a formal power series $f \in {\bf k}[[X]]$ is called a {\bf multi-quasisymmetric function} if, for each multi-composition ${\bf w}=(w_1, \ldots, w_k)$, the coefficients of monomials $x_{j_1}^{w_1} \cdots x_{j_k}^{w_k}$ in $f$ are equal for all positive integers $1 \leq j_1 \leq \cdots \leq j_k$. Denote the set of all multi-quasisymmetric functions by $\mqsym$.
\mlabel{def:instrictdef}
\end{defn}

In analog to the monomial basis for quasisymmetric functions, we next give the definition of monomial multi-quasisymmetric functions and show that they form a basis of $\mqsym$.

\begin{defn}\mlabel{def:mmqf}
For a multi-composition ${\bf w}=(w_1, \ldots, w_k)$ of length $k$, the corresponding {\bf monomial multi-quasisymmetric function}
$M_{{\bf w}}$ is defined by
\begin{align}
M_{{\bf w}} := \sum \limits_{ {\bf j}= \{ j_1 < \cdots < j_k \}} x^{{\bf w}}_{{\bf j}},
\label{eq:mmqf}
\end{align}
where ${\bf j}$ runs through all sets of $k$ positive integers.
\end{defn}

\begin{remark}
When $m=2$, the notions of multi-compositions, multi-quasisymmetric functions and monomial multi-quasisymmetric functions recover, respectively, the notions of bicompositions, diagonally quasisymmetric function and bimonomial quasisymmetric functions in~\mcite{ABB06}.
\end{remark}

\begin{prop}
The set $\{M_{\bf w} \mid {\bf w} \in \mcomp \}$ is a $\bf k$-basis of $\mqsym$.
\mlabel{prop:kbasis}
\end{prop}

\begin{proof}
By Definition~\mref{def:instrictdef}, each $f\in \mqsym$ is a ${\bf k}$-linear combination of $\{M_{\bf w} \mid {\bf w} \in \mcomp \}$.
We only need to show that $M_{\bf w}$ with  ${\bf w} \in \mcomp $ are linearly independent. Assume $\sum \limits_{{\bf w} \in S} c_{\bf w} M_{\bf w}=0$, where $S$ is a finite set of multi-compositions and $c_{\bf w} \in {\bf k}$. Note that $\mqsym \subset {\bf k}[[\X]]$ and the set of all monomials is a ${\bf k}$-basis for ${\bf k}[[\X]]$. By Eq.~(\ref{eq:mmqf}) for $M_{\bf w}$, we have
\begin{align*}
\sum \limits_{{\bf w} \in S} c_{\bf w} x_{\bf j}^{\bf w}=0, \text{ where ${\bf j}=\{1<2<\cdots<\ell({\bf w}) \}$},
\end{align*}
and so $c_{\bf w}=0$ for each ${\bf w} \in S$. This completes the proof.
\end{proof}

\begin{theorem}
$\mqsym$ is a subalgebra of $\bfk[[\X]]$, called the {\bf algebra of multi-quasisymmetric functions}. Moreover, $\mqsym$ is a Hopf algebra.
\mlabel{pp:mqsym}
\end{theorem}
This property of multi-quasisymmetric functions will be obtained as a special case of the more general notion of multi-quasisymmetric functions with semigroup exponents, to be introduced in Section~\mref{ss:mqsymqs}. More precisely, the first conclusion follows from Corollary~\mref{co:eqsymsubalg} (see Remark~\mref{rk:ecases}) and the second conclusion follows from Proposition~\mref{prop:hopf1}.

\subsection{Fundamental multi-quasisymmetric functions}
\mlabel{ss:gm}
In this subsection, we first give the definition of fundamental multi-quasisymmetric functions. Then we show that all fundamental multi-quasisymmetric functions form another $\bf k$-basis of $\mqsym$.

Assume ${\bf c}$ is an $[m]$-composition as follows
\begin{equation}
{\bf c}=({\bf c}_1,\ldots,{\bf c}_k)=
\begin{pmatrix}
a_{1,1} & \ldots & a_{1,k}\\
a_{2,1} & \ldots & a_{2,k}\\
\vdots & \ddots & \vdots \\
a_{m,1} & \ldots & a_{m,k}\\
\end{pmatrix}.
\mlabel{eq:cmat}
\end{equation}
Using the alphabet $\{x_{1}, x_{2}, \ldots, x_{m} \}$, each column ${\bf c}_i=(a_{1,i},\ldots, a_{m,i})^T$ is associated a word $$w_{{\bf c}_i}:=x_{1}^{a_{1,i}} \cdots x_{m}^{a_{m,i}}.$$
Then $w_{\bf c} := \Pi_{i=1}^k w_{{\bf c}_i} $. We give an example as a demonstration.

\begin{exam}
For $[3]$-compositions ${\bf c}=\begin{pmatrix} 1 \\ 2 \\ 1 \\ \end{pmatrix}$ and ${\bf c'} = \begin{pmatrix}
		0 & 2\\
		1 & 0\\
		0 & 1
	\end{pmatrix}$,
the words associated to ${\bf c}$ and ${\bf c'}$ are
\begin{align*}
w_{\bf c}=x_{1} x_{2}^2 x_{3} \text{ and } \, w_{\bf c'}=x_{2} x_{1}^2 x_{3}.
\end{align*}
\end{exam}

The {\bf fundamental multi-quasisymmetric function of index ${\bf c}$} is defined to be
\begin{align}
F_{\bf c} := \sum_f \prod_{x_i \in w_{\bf c}} x_{i, f(x_i)},
\label{eq:fqs}
\end{align}
where the sum is taken over all maps $f$ from the sequence of letters of the word $w_{\bf c}$ to $\PP$ such that $f$ is weakly increasing inside the letters of each column ${\bf c}_i$ and strictly increasing between the letters of two columns,
and $x_i \in w_{\bf c}$ means $x_i$ are letters of $w_{\bf c}$.

\begin{exam} For the $[3]$-compositions ${\bf c}=\begin{pmatrix} 1 \\ 2 \\ 1 \\ \end{pmatrix}$ and ${\bf c'} = \begin{pmatrix}
		0 & 2\\
		1 & 0\\
		0 & 1
	\end{pmatrix}$, we have
\begin{align*}
F_{\bf c}=\sum \limits_{i_1 \leq i_2 \leq i_3 \leq i_4} x_{1, i_1} x_{2, i_2} x_{2, i_3} x_{3, i_4} \, \text{ and } F_{\bf c'}= \sum \limits_{i_1 < i_2 \leq i_3 \leq i_4} x_{2, i_1} x_{1, i_2} x_{1, i_3} x_{3, i_4}.
\end{align*}
\mlabel{exam:fqs}
\end{exam}

\begin{remark}
With a change of notations, the notion of fundamental multi-quasisymmetric functions agrees with the fundamental $G^m$-quasisymmetric functions in~\mcite{AC16}, where a fundamental $G^m$-quasisymmetric function is defined for an $m$-composition and an $m$-composition can be viewed as a matrix corresponding to an $[m]$-composition in our paper.
\end{remark}

To establish the connection between the fundamental multi-quasisymmetric functions and the multi-quasisymmetric functions, we introduce the following notations.

Let ${\bf c}=({\bf c}_1, \ldots, {\bf c}_k)$ be an $[m]$-composition as in Eq.~(\ref{eq:cmat}). Denote $|{\bf c}_i| :=\sum \limits_{j=1}^m a_{j,i}$ for $1 \leq i \leq k$ and $|{\bf c}|=\sum \limits_{i=1}^k |{\bf c}_i|$. Then define the {\bf subnorm} $\Des({\bf c})$  of the $[m]$-composition ${\bf c}=({\bf c}_1, \ldots, {\bf c}_k)$ by
\begin{align*}
\Des({\bf c})=\Big\{ |{\bf c}_1|, |{\bf c}_1|+|{\bf c}_2|, \ldots, \sum \limits_{i=1}^{k-1} |{\bf c}_i| \Big\}.
\end{align*}
In particular, $\Des({\bf c})=\emptyset$ when $k=1$.

Note that for any positive integer $j \in [|{\bf c}|]$ (that is, $j\leq |{\bf c}|)$, there are unique $r_j, \ell_j$ with $0 \leq r_j \leq k-1$ and $1 \leq \ell_j \leq m$ such that
\begin{align}
\sum \limits_{p=1}^{r_j}|{\bf c}_p|+\sum \limits_{q=1}^{\ell_j-1}a_{q, r_j+1} < j \leq \sum \limits_{p=1}^{r_j}|{\bf c}_p|+\sum \limits_{q=1}^{\ell_j}a_{q, r_j+1}.
\mlabel{eq:location}
\end{align}
In other words,
$$ r_j:=\max\{ r\in [k-1]\cup \{0\}\,|\, j> |{\bf c}_1|+\cdots+|{\bf c}_r|\},  \quad
\ell_j:=\min\Big\{\ell \in [m]\,|\, j\leq \sum \limits_{p=1}^{r_j}|{\bf c}_p|+\sum \limits_{q=1}^{\ell}a_{q, r_j+1} \Big\}.
$$
Associated to the $[m]$-composition ${\bf c}=({\bf c}_1, \ldots, {\bf c}_k)$, define a function
\begin{align}
g_{\bf c}: [|{\bf c}|] \rightarrow [m], \quad  j \mapsto \ell_j,
\mlabel{eq:locationfunction}
\end{align}
where $\ell_j$ is the unique $\ell_j$ satisfying condition~(\ref{eq:location}). In other words, $g_{\bf c}$ maps the integers between
$$\sum \limits_{p=1}^{r_j}|{\bf c}_p|+\sum \limits_{q=1}^{\ell_j-1}a_{q, r_j+1}+1\,\text{ and }\,\sum \limits_{p=1}^{r_j}|{\bf c}_p|+\sum \limits_{q=1}^{\ell_j}a_{q, r_j+1}$$ to $\ell_j$, where $a_{\ell_j, r_{j}+1}$ is at the $\ell_j$-th row in the matrix ${\bf c}$.  Then by the form of Eq.~(\ref{eq:fqs}), we give another expression of the fundamental multi-quasisymmetric function of index ${\bf c}$ as
\begin{align}
F_{\bf c} =\sum_{\substack{ 1 \leq i_1 \leq i_2 \leq \cdots  \leq  i_{|{\bf c}|},  \\ \text{$i_k <i_{k+1}$ if $k \in \Des({\bf c})$}}} \prod_{j=1}^{|{\bf c}|} x_{g_{\bf c}(j), {i_j}}.
\mlabel{eq:newfqs}
\end{align}

\begin{exam}
For the $[3]$-compositions ${\bf c}=\begin{pmatrix} 1 \\ 2 \\ 1 \\ \end{pmatrix}$ and ${\bf c'} = \begin{pmatrix}
		0 & 2\\
		1 & 0\\
		0 & 1
	\end{pmatrix}$, we have $\Des({\bf c})=\emptyset$ and $\Des({\bf c'})=\{ 1\}$. The function $g_{\bf c}$ and $g_{{\bf c'}}$ are defined as
\begin{align*}
 g_{\bf c}:&\ [4] \rightarrow [3],\quad  g_{\bf c}(1)=1,\, g_{\bf c}(2)=2, \,g_{\bf c}(3)=2, \,g_{\bf c}(4)=3, \\
 g_{{\bf c'}}:&\ [4] \rightarrow [3],\quad g_{{\bf c'}}(1)=2,\, g_{{\bf c'}}(2)=1,\, g_{{\bf c'}}(3)=1,\, g_{{\bf c'}}(4)=3.
\end{align*}
Then by Eq.~(\ref{eq:newfqs}), we also have
\begin{align*}
F_{\bf c}=\sum \limits_{i_1 \leq i_2 \leq i_3 \leq i_4} x_{1, i_1} x_{2, i_2} x_{2, i_3} x_{3, i_4} \, \text{ and } F_{{\bf c'}}= \sum \limits_{i_1 < i_2 \leq i_3 \leq i_4} x_{2, i_1} x_{1, i_2} x_{1, i_3} x_{3, i_4},
\end{align*}
which agree with Example~\ref{exam:fqs}.
\end{exam}

Let ${\bf c}=({\bf c}_1, \ldots, {\bf c}_k)$ be an $[m]$-composition as in Eq.~(\ref{eq:cmat}). If $|{\bf c}|=n$, then ${\bf c}$ is an $[m]$-composition of $n$. Denote by $\mcompn$ the set of all $[m]$-compositions of $n$. Next we define a partial order $\opare$ on $\mcompn$ as follows.

First for ${\bf c}, {\bf c'} \in \mcompn$, define ${\bf c} < {\bf c'}$ if ${\bf c}=({\bf c}_1, \ldots, {\bf c}_k)$ and there is a column
${\bf c}_i=(a_{1,i}, \ldots, a_{m,i})^{T}$ of ${\bf c}$ and an entry $a_{j,i}$ of ${\bf c}_i$, which is decomposed as $a_{j,i} = a'_{j,i}  + a''_{j,i}$ with $a'_{j,i},a''_{j,i}\in \NN$ and at least one of $a'_{j,i},a''_{j,i}$ being positive; while $${\bf c'}=({\bf c}_1, \ldots, {\bf c}_{i-1}, {\bf c}'_i, {\bf c}'_{i+1}, {\bf c}_{i+1}, \ldots, {\bf c}_k),$$
where $${\bf c}'_i=(a_{1,i}, \ldots, a_{j-1, i}, a'_{j,i}, 0, \ldots, 0)^T\,\text{ and }\, {\bf c}'_{i+1}=(0, \ldots, 0, a''_{j,i}, a_{j+1,i}, \ldots, a_{m,i})^T.$$
So the $i$-th column of ${\bf c}$ is broken into two columns of ${\bf c'}$.

Then define $\opar$ to be the transitive closure of the relation $<$. In other words, for ${\bf c}, {\bf c'} \in \mcompn$, ${\bf c} \opar {\bf c'}$ if and only if there are ${\bf w}_0, {\bf w}_1, \ldots, {\bf w}_k \in \mcompn$ such that
\begin{align*}
{\bf c}={\bf w}_0 < {\bf w}_1 < \cdots < {\bf w}_k={\bf c'}.
\end{align*}
Also let $\opare$ denote the transitive and reflexive closure of $<$. Thus $\opare$ is the smallest poset containing $<$.

\begin{exam}
For the [2]-composition ${\bf c}=\begin{pmatrix}
1 \\
2
\end{pmatrix}$ of length 1,
those ${\bf c'}$ with ${\bf c} \opar {\bf c'}$ are given by
\begin{align*}
\begin{pmatrix}
1 \\
2
\end{pmatrix} < \begin{pmatrix}
1 & 0\\
0& 2
\end{pmatrix} < \begin{pmatrix}
1 & 0 & 0\\
0 & 1 & 1
\end{pmatrix}~ \text{ and }~
\begin{pmatrix}
1 \\
2
\end{pmatrix} <
\begin{pmatrix}
1 & 0\\
1 & 1
\end{pmatrix} < \begin{pmatrix}
1 & 0 & 0\\
0 & 1 & 1
\end{pmatrix}.
\end{align*}
\mlabel{ex:2comp}
\end{exam}

We give an equivalent characterization of the relation $\opar$.

\begin{lemma}
For two $[m]$-compositions ${\bf c}, {\bf c'} \in \mcompn$, ${\bf c} \opar {\bf c'}$ if and only if $\Des({\bf c}) \subsetneq \Des({\bf c'})$ and $g_{\bf c}=g_{\bf c'}$, where $g_{\bf c}$ and $g_{\bf c'}$ are the functions defined by Eq.~\eqref{eq:locationfunction}.
\mlabel{lem:equichar}
\end{lemma}

\begin{proof}
Define another transitive relation $\opar_1$ on $\mcompn$ as follows. For two $[m]$-compositions ${\bf c}, {\bf c'} \in \mcompn$, define ${\bf c}<_1 {\bf c'}$ if  $\Des({\bf c'})$ has one more element than $ \Des({\bf c})$ and $g_{\bf c}=g_{\bf c'}$. Then define $\opar_1$ to be the transitive closure of $<_1$.

Then we only need to prove that ${\bf c} \opar {\bf c'}$ if and only if ${\bf c} \opar_1 {\bf c'}$. As $\opar$ and $\opar_1$ are the transitive closures of $<$ and $<_1$ respectively, it suffices to prove that ${\bf c} < {\bf c'}$ if and only if ${\bf c}<_1{\bf c'}$.

Fix an $[m]$-composition ${\bf c}$ as follows
\begin{equation*}
{\bf c}=({\bf c}_1,\ldots,{\bf c}_k)=
\begin{pmatrix}
a_{1,1} & \ldots & a_{1,k}\\
a_{2,1} & \ldots & a_{2,k}\\
\vdots & \ddots & \vdots \\
a_{m,1} & \ldots & a_{m,k}\\
\end{pmatrix}
\end{equation*}
and assume that ${\bf c}< {\bf c'}$.
Then by the definition of the relation $<$, ${\bf c'}$ is obtained from ${\bf c}$ by breaking one column ${\bf c}_i=(a_{1,i},a_{2,i},\cdots, a_{m,i})^{T}$ of ${\bf c}$ into two columns ${\bf c}'_i, {\bf c}'_{i+1}$ of ${\bf c'}$ with
$${\bf c}'_i=(a_{1,i}, \ldots, a_{j-1, i}, a'_{j,i}, 0, \ldots, 0)^T\,\text{ and }\, {\bf c}'_{i+1}=(0, \ldots, 0, a''_{j,i}, a_{j+1,i}, \ldots, a_{m,i})^T,$$
where $a_{j,i} = a'_{j,i} + a''_{j,i}$ with $a'_{j,i},a''_{j,i}\in \NN$ and at least one of $a'_{j,i},a''_{j,i}$ is nonzero.
Thus $|{\bf c}_i'|+|{\bf c}'_{i+1}|=|{\bf c}_i|$ and  $\Des({\bf c'})=\Des({\bf c}) \sqcup \{ |{\bf c'}_i| \}$. Also, by the definitions of $g_{\bf c}$ and $g_{\bf c'}$ in Eq.~(\ref{eq:locationfunction}), we have $g_{\bf c}=g_{\bf c'}$. Hence ${\bf c} <_1 {\bf c'}$.

Conversely, still fix an $[m]$-composition ${\bf c}$ as follows
\begin{equation*}
{\bf c}=({\bf c}_1,\ldots,{\bf c}_k)=
\begin{pmatrix}
a_{1,1} & \ldots & a_{1,k}\\
a_{2,1} & \ldots & a_{2,k}\\
\vdots & \ddots & \vdots \\
a_{m,1} & \ldots & a_{m,k}\\
\end{pmatrix}
\end{equation*}
and assume ${\bf c} <_1 {\bf c'}$, i.e., $\Des({\bf c'})$ has one more element than $\Des({\bf c})$ and $g_{\bf c'}=g_{\bf c}$. Suppose that the extra element of $\Des({\bf c'})$ is $s$: $\Des({\bf c'})=\Des({\bf c}) \sqcup \{ s\}$.  As $\Des({\bf c'})$ has $k$ elements, the corresponding matrix ${\bf c'}$ has $k+1$ columns. Note that for the element $s\in \Des({\bf c'})\setminus \Des({\bf c})$, there are unique $r_s, \ell_s$ with $0 \leq r_s \leq k-1$ and $1 \leq \ell_s \leq m$ such that
\begin{align*}
\sum \limits_{p=1}^{r_s}|{\bf c}_p|+\sum \limits_{q=1}^{\ell_s-1}a_{q, r_s+1} < s < \sum \limits_{p=1}^{r_s}|{\bf c}_p|+\sum \limits_{q=1}^{\ell_s}a_{q, r_s+1}.
\end{align*}
Then $\Des({\bf c'})=\{ |{\bf c}_1|, \ldots, \sum \limits_{i=1}^{r_s} |{\bf c}_i|, s, \sum \limits_{i=1}^{r_s+1} |{\bf c}_i|, \ldots, \sum \limits_{i=1}^{k-1} |{\bf c}_i| \}$. Notice that $g_{\bf c}$ maps the integers between
$$\sum \limits_{p=1}^{r_j}|{\bf c}_p|+\sum \limits_{q=1}^{\ell_j-1}a_{q, r_j+1}+1\,\text{ and }\, \sum \limits_{p=1}^{r_j}|{\bf c}_p|+\sum \limits_{q=1}^{\ell_j}a_{q, r_j+1}$$ to $\ell_j$, where $a_{\ell_j, r_{j}+1}$ is at the $\ell_j$-th row in the matrix ${\bf c}$. By $g_{\bf c}=g_{\bf c'}$ and
$$\Des({\bf c'})=\{ |{\bf c}_1|, \ldots, \sum \limits_{i=1}^{r_n} |{\bf c}_i|, n, \sum \limits_{i=1}^{r_n+1} |{\bf c}_i|, \ldots, \sum \limits_{i=1}^{k} |{\bf c}_i| \},$$
we have
{\small
\begin{equation*}
{\bf c'}=({\bf c}_1,\ldots, {\bf c}_{r_s}, {\bf c'}_{r_s+1}, {\bf c'}_{r_s+2}, {\bf c}_{r_s+2}, \ldots, {\bf c}_{k})=
\begin{pmatrix}
a_{1,1} & \ldots & a_{1,r_s} & a'_{1,r_s+1} & a'_{1,r_s+2} & a_{1, r_s+2} & \ldots & a_{1,k}\\
a_{2,1} & \ldots & a_{2, r_s} & a'_{2,r_s+1} & a'_{2,r_s+2} & a_{2, r_s+2} & \ldots & a_{2,k}\\
\vdots & \ddots & \vdots & \vdots & \vdots  & \vdots & \ddots & \vdots \\
a_{m,1} & \ldots & a_{m, r_s} & a'_{m,r_s+1} & a'_{m,r_s+2} & a_{m, r_s+2} & \ldots & a_{m,k}\\
\end{pmatrix},
\end{equation*}
}
with
\begin{align*}
a'_{j, r_s+1}=&\ a_{j,r_s+1}, \quad  1 \leq j \leq \ell_s-1,\\
a'_{\ell_s, r_s+1}=&\ s-\left(\sum \limits_{p=1}^{r_s}|{\bf c}_p| \right)-\left( \sum \limits_{q=1}^{\ell_s-1}a_{q, r_s+1} \right), \quad  a'_{j, r_s+1}=0, \quad  \ell_s+1 \leq j \leq m,\\
a'_{j,r_s+2}=&\ 0, \quad  1 \leq j \leq \ell_s-1, \\
a'_{\ell_s, r_s+2}=&\ \sum \limits_{p=1}^{r_s}|{\bf c}_p|+\sum \limits_{q=1}^{\ell_s}a_{q, r_s+1}-s,\quad a'_{j,r_s+2}=a_{j,r_s+1}, \quad  \ell_s+1 \leq j \leq m.
\end{align*}
This means that ${\bf c'}$ is obtained from ${\bf c}$ by breaking the $(r_s+1)$-th column of ${\bf c}$ into the $(r_s+1)$-th column and $(r_s+2)$-th column of ${\bf c'}$. Hence ${\bf c} < {\bf c'}$.
This completes the proof.
\end{proof}

Let ${\bf c}=({\bf c}_1, \ldots, {\bf c}_k)\in \mcompn$. As $|{\bf c}|=n$ and each column vector ${\bf c}_i$ is nonzero, we have $\Des({\bf c}) \subseteq [n-1]$.

\begin{lemma}
For an $[m]$-composition \begin{equation*}
{\bf c}=({\bf c}_1,\ldots,{\bf c}_k)=
\begin{pmatrix}
a_{1,1} & \ldots & a_{1,k}\\
a_{2,1} & \ldots & a_{2,k}\\
\vdots & \ddots & \vdots \\
a_{m,1} & \ldots & a_{m,k}\\
\end{pmatrix},
\end{equation*}
if ${\bf c}$ is in $\mcompn$ and if there is a set $Z$ such that $\Des({\bf c}) \subseteq Z \subseteq [n-1] $, then there is a unique $[m]$-composition ${\bf c'} \in \mcompn$ such that $\Des({\bf c'})=Z$ and $g_{\bf c}=g_{{\bf c'}}$.
\mlabel{lem:exactone}
\end{lemma}

\begin{proof}
Assume $Z\setminus \Des({\bf c})=\{ s_1< s_2< \cdots < s_t\}$.  We prove the conclusion by induction on $t\geq 1$. If $t=1$, then similar to the proof of the converse part of Lemma~\ref{lem:equichar}, there is an $[m]$-composition ${\bf w}_1 \in \mcompn$ with $\Des({\bf w}_1)=\Des({\bf c}) \sqcup \{s_1 \}$ and $g_{{\bf w}_1}=g_{{\bf c}}$, where ${\bf w}_1$ is obtained by breaking the $(r_{s_1}+1)$-th column of ${\bf c}$ into the $(r_{s_1}+1)$-th column and the $(r_{s_1}+2)$-th column of ${\bf w}_1$. Taking ${\bf c'}={\bf w}_1$, then ${\bf c'}$ is the unique $[m]$-composition ${\bf c'} \in \mcompn$ such that $\Des({\bf c'})=\Des({\bf c}) \sqcup \{ s_1\}$ and $g_{\bf c}=g_{{\bf c'}}$.

If $t=2$, then similarly there is an $[m]$-composition ${\bf w}_2 \in \mcompn$ with $\Des({\bf w}_2)=\Des({\bf w}_1) \sqcup \{s_2 \}=\Des({\bf c}) \sqcup \{s_1,s_2 \}$ and $g_{{\bf w}_2}=g_{{\bf w}_1}=g_{{\bf c}}$. However, we could also have two $[m]$-compositions ${\bf w}'_1, {\bf w}'_2 \in \mcompn$ with $\Des({\bf w}'_2)=\Des({\bf w}'_1) \sqcup \{s_1 \}=\Des({\bf c}) \sqcup \{s_1, s_2 \}$ and $g_{{\bf w}'_2}=g_{{\bf w}'_1}=g_{{\bf c}}$. Next we show that ${\bf w}_2$ and ${\bf w}'_2$ are the same.

Note that the ${\bf w}_2$ is obtained by first breaking ${\bf c}_{r_{s_1}+1}$ into two columns and then breaking ${\bf c}_{r_{s_2}+1}$ into two columns. The ${\bf w}'_2$ is obtained by first breaking ${\bf c}_{r_{s_2}+1}$ into two columns and then breaking ${\bf c}_{r_{s_1}+1}$ into two columns. If $r_{s_1} \neq r_{s_2}$, then the two operations induce the same matrix and so ${\bf w}_2={\bf w}'_2$. If $r_{s_1} = r_{s_2}$, without loss of generality, assume $r_{s_1}=r_{s_2}=0$ and we compare $\ell_{s_1}, \ell_{s_2}$. As $s_1< s_2$ and $r_{s_1}=r_{s_2}$, we have $\ell_{s_1} \leq \ell_{s_2}$. If $\ell_{s_1}< \ell_{s_2}$, without the loss of generality, we may assume $\ell_{s_1}=1$ and $\ell_{s_2}=2$, i.e., $0<s_1 \leq a_{1,1}$ and $a_{1,1} \leq s_2 \leq a_{2,1}$. Then
\begin{equation*}
{\bf w}_1=
\begin{pmatrix}
s_1 & a_{1,1}-s_1 & a_{1,2} & \ldots & a_{1,k}\\
0 & a_{2,1} & a_{2,2} & \ldots & a_{2,k}\\
0 & a_{3,1} & a_{3,2} & \ldots & a_{3,k} \\
\vdots & \vdots & \vdots & \ddots & \vdots \\
0 & a_{m,1} & a_{m,2} & \ldots & a_{m,k}\\
\end{pmatrix},
{\bf w}_2=
\begin{pmatrix}
s_1 & a_{1,1}-s_1 & 0 & a_{1,2} & \ldots & a_{1,k}\\
0 & s_2-a_{1,1} & a_{1,1}+a_{2,1}-s_2 & a_{2,2} & \ldots & a_{2,k}\\
0&  0 & a_{3,1} & a_{3,2} & \ldots & a_{3,k} \\
\vdots & \vdots & \vdots & \vdots & \ddots & \vdots \\
0 & 0 & a_{m,1} & a_{m,2} & \ldots & a_{m,k}\\
\end{pmatrix}
\end{equation*}
and
\begin{equation*}
{\bf w}'_1=
\begin{pmatrix}
a_{1,1} & 0 & a_{1,2} & \ldots & a_{1,k}\\
s_2-a_{1,1} & a_{1,1}+a_{2,1}-s_2 & a_{2,2} & \ldots & a_{2,k}\\
0 & a_{3,1} & a_{3,2} & \ldots & a_{3,k} \\
\vdots & \vdots & \vdots & \ddots & \vdots \\
0 & a_{m,1} & a_{m,2} & \ldots & a_{m,k}\\
\end{pmatrix},
\end{equation*}
\begin{equation*}
{\bf w}'_2=
\begin{pmatrix}
s_1 & a_{1,1}-s_1 & 0 & a_{1,2} & \ldots & a_{1,k}\\
0 & s_2-a_{1,1} & a_{1,1}+a_{2,1}-s_2 & a_{2,2} & \ldots & a_{2,k}\\
0&  0 & a_{3,1} & a_{3,2} & \ldots & a_{3,k} \\
\vdots & \vdots & \vdots & \vdots & \ddots & \vdots \\
0 & 0 & a_{m,1} & a_{m,2} & \ldots & a_{m,k}\\
\end{pmatrix}.
\end{equation*}
So ${\bf w}_2={\bf w}'_2$. If $\ell_{s_1}= \ell_{s_2}$, without loss of generality, we may assume $\ell_{s_1}=\ell_{s_2}=1$, i.e., $0<s_1 < s_2 \leq a_{1,1}$. Then
\begin{equation*}
{\bf w}_1=
\begin{pmatrix}
s_1 & a_{1,1}-s_1 & a_{1,2} & \ldots & a_{1,k}\\
0 & a_{2,1} & a_{2,2} & \ldots & a_{2,k}\\
0 & a_{3,1} & a_{3,2} & \ldots & a_{3,k} \\
\vdots & \vdots & \vdots & \ddots & \vdots \\
0 & a_{m,1} & a_{m,2} & \ldots & a_{m,k}\\
\end{pmatrix},
{\bf w}_2=
\begin{pmatrix}
s_1 & s_2-s_1 & a_{1,1}-s_2 & a_{1,2} & \ldots & a_{1,k}\\
0 & 0 & a_{2,1} & a_{2,2} & \ldots & a_{2,k}\\
0&  0 & a_{3,1} & a_{3,2} & \ldots & a_{3,k} \\
\vdots & \vdots & \vdots & \vdots & \ddots & \vdots \\
0 & 0 & a_{m,1} & a_{m,2} & \ldots & a_{m,k}\\
\end{pmatrix}
\end{equation*}
and
\begin{equation*}
{\bf w}'_1=
\begin{pmatrix}
s_2 & a_{1,1}-s_2 & a_{1,2} & \ldots & a_{1,k}\\
0 & a_{2,1} & a_{2,2} & \ldots & a_{2,k}\\
0 & a_{3,1} & a_{3,2} & \ldots & a_{3,k} \\
\vdots & \vdots & \vdots & \ddots & \vdots \\
0 & a_{m,1} & a_{m,2} & \ldots & a_{m,k}\\
\end{pmatrix},
\end{equation*}
\begin{equation*}
{\bf w}'_2=
\begin{pmatrix}
s_1 & s_2-s_1 & a_{1,1}-s_2 & a_{1,2} & \ldots & a_{1,k}\\
0 & 0 & a_{2,1} & a_{2,2} & \ldots & a_{2,k}\\
0&  0 & a_{3,1} & a_{3,2} & \ldots & a_{3,k} \\
\vdots & \vdots & \vdots & \vdots & \ddots & \vdots \\
0 & 0 & a_{m,1} & a_{m,2} & \ldots & a_{m,k}\\
\end{pmatrix}.
\end{equation*}
So ${\bf w}_2={\bf w}'_2$. Take ${\bf c}'={\bf w}_2$, hence ${\bf c'}$ is the unique $[m]$-composition ${\bf c'} \in \mcompn$ such that $\Des({\bf c'})=\Des({\bf c}) \sqcup \{ s_1,s_2\}$ and $g_{\bf c}=g_{{\bf c'}}$.

For the inductive step, suppose the conclusion holds when $t \leq N$ for $N\geq 1$ and consider the case of $t=N+1$.
Now $Z \setminus \Des({\bf c})=\{ s_1< s_2< \cdots < s_{N+1}\}$. Take any $i, j$ such that $1 \leq i <j \leq N+1$, by the induction hypothesis, there are ${\bf w}_{N}, {\bf w}'_{N} \in \mcompn$ such that
\begin{align*}
\Des({\bf w}_{N})=&\ \Des({\bf c}) \sqcup \{s_1< \cdots< s_{i-1}< \hat{s}_i<s_{i+1}< \cdots <s_{N+1} \},\\
\Des({\bf w}'_{N})=&\ \Des({\bf c}) \sqcup \{s_1< \cdots< s_{j-1}< \hat{s}_j<s_{j+1}< \cdots <s_{N+1} \},\\
g_{{\bf w}_{N}}=&\ g_{\bf c}=g_{{\bf w}'_{N}},
\end{align*}
where $\hat{s}_i$ means that $s_i$ is omitted. Then as in the case of $t=1$, there are ${\bf w}_{N+1}, {\bf w}'_{N+1} \in \mcompn$ such that
\[
\Des({\bf w}_{N+1})=\Des({\bf w}_{N}) \sqcup \{s_i \},\, \Des({\bf w}'_{N+1})=\Des({\bf w}'_{N}) \sqcup \{ s_{j}\}\,\text{ and }\, g_{{\bf w}_{N+1}}=g_{{\bf w}_{N}}=g_{{\bf c}}=g_{{\bf w}'_{N}}=g_{{\bf w}'_{N+1}}.
 \]
Next we prove that ${\bf w}_{N+1}={\bf w}'_{N+1}$. By the induction hypothesis, there is a unique ${\bf w}_{N-1} \in \mcompn$ such that
$$\Des({\bf w}_{N-1})=\Des({\bf c}) \sqcup \{s_1< \cdots <s_{i-1}< \hat{s}_{i}< s_{i+1}< \cdots< s_{j-1}< \hat{s}_j <s_{j+1}< \cdots < s_{N+1}\}
$$
and $g_{{\bf w}_{N-1}}=g_{{\bf c}}$. Similar to the case of $t=2$, there is a unique ${\bf w}''_{N+1} \in \mcompn$ such that $\Des({\bf w}''_{N+1})=\Des({\bf w}_{N-1}) \sqcup \{ s_i, s_j\}$ and $g_{{\bf w}''_{N+1}}=g_{{\bf w}_{N-1}}=g_{\bf c}$. Hence ${\bf w}_{N+1}={\bf w}''_{N+1}={\bf w}'_{N+1}$. By the arbitrariness of $i, j$,  there is a unique $[m]$-composition ${\bf c'} \in \mcompn$ such that $\Des({\bf c'})=\Des({\bf c}) \sqcup \{ s_1<\cdots < s_{N+1}\}$ and $g_{\bf c}=g_{{\bf c'}}$. This completes the proof.
\end{proof}

\begin{remark}
For an $[m]$-composition ${\bf w}=({\bf w}_1, {\bf w}_2, \ldots, {\bf w}_k)$ as in Eq.~(\ref{eq:cmat}), the monomial multi-quasisymmetric function $M_{{\bf w}}$ defined in Definition~\mref{def:mmqf} is
\begin{align*}
M_{{\bf w}} = \sum \limits_{ {\bf j}= \{ j_1 < \cdots < j_k \}} x^{{\bf w}}_{{\bf j}},
\end{align*}
where ${\bf j}$ runs through all sets of $k$ positive integers. Actually we have another expression of $M_{{\bf w}}$ given by
\begin{align}
 M_{{\bf w}}=\sum_{\substack{i_1 \leq i_2 \leq \cdots \leq i_{|{\bf w}|}, \\ \text{ $i_j < i_{j+1}$ iff $j \in \Des(w)$}}}  \prod_{j=1}^{|{\bf c}|} x_{g_c(j), {i_j}}.
\mlabel{eq:nw}
\end{align}
\end{remark}

A basic property of quasisymmetric functions is that the fundamental quasisymmetric functions $F_\alpha$ are linear combinations of the monomial quasisymmetric functions $M_\alpha$, as $\alpha$ runs through the compositions~\cite{GR14}. We next generalize this property to multi-quasisymmetric functions.
We first use Example~\mref{ex:2comp} to illustrate how the general result works.

\begin{exam}
For the [2]-composition ${\bf c}=\begin{pmatrix}
1 \\
2
\end{pmatrix}$ of length 1, we obtain
\begin{align*}
F_{{\bf c}}=&\ \sum \limits_{i \leq j \leq k} x_{1,i} x_{2,j} x_{2,k}\\
=&\ \sum \limits_{i = j = k} x_{1,i} x_{2,j} x_{2,k}+\sum \limits_{i < j = k} x_{1,i} x_{2,j} x_{2,k}+\sum \limits_{i = j < k} x_{1,i} x_{2,j} x_{2,k}+\sum \limits_{i < j < k} x_{1,i} x_{2,j} x_{2,k}\\
=&\ M_{\begin{pmatrix} 1 \\  2 \end{pmatrix}}+ M_{\begin{pmatrix} 1 & 0\\0&  2        \end{pmatrix}}+M_{\begin{pmatrix} 1 & 0\\ 1& 1        \end{pmatrix}}+M_{\begin{pmatrix} 1 & 0 & 0 \\ 0& 1 &1         \end{pmatrix}}.
\end{align*}
\end{exam}

\begin{prop}
For an $[m]$-composition ${\bf c} \in \mcompn$, we have
\begin{align*}
F_{{\bf c}}=\sum \limits_{{\bf c} \opare {\bf c}'} M_{{\bf c}'}
\label{eq:relation}
\end{align*}
and hence $F_{\bf c}$ is in $\mqsym.$
\mlabel{prop:mcom}
\end{prop}

\begin{proof}
For ${\bf c} \in \mcompn$, we have
\begin{align*}
F_{\bf c} =&\ \sum_{\substack{ 1 \leq i_1 \leq i_2 \leq \cdots  \leq  i_{|{\bf c}|},  \\ \text{$i_k <i_{k+1}$ if $k \in \Des({\bf c})$}}} \prod_{j=1}^{|{\bf c}|} x_{g_{\bf c}(j), {i_j}} \quad \text{(by Eq.~(\ref{eq:newfqs}))}\\
=&\ \sum_{\substack{Z \subseteq [|{\bf c }|-1], \\ \text{ $\Des({\bf c}) \subseteq Z$} }} \sum_{\substack{ 1 \leq i_1 \leq i_2 \leq \cdots  \leq  i_{|{\bf c}|},  \\ \text{$i_k <i_{k+1}$ iff $k \in Z$}}} \prod_{j=1}^{|{\bf c}|} x_{g_{\bf c}(j), {i_j}} \quad \text{(by dividing $i_j \leq i_{j+1}$ to $i_j < i_{j+1}$ and $i_j=i_{j+1}$)}\\
=&\ \sum_{\substack{{\bf c'} \in \mcompn, \\ \text{ $\Des({\bf c}) \subseteq \Des({\bf c'})$ and $g_{\bf c}=g_{{\bf c'}}$} }} \sum_{\substack{ 1 \leq i_1 \leq i_2 \leq \cdots  \leq  i_{|{\bf c}|},  \\ \text{$i_k <i_{k+1}$ iff $k \in \Des({\bf c'})$}}} \prod_{j=1}^{|{\bf c}|} x_{g_{{\bf c'}}(j), {i_j}} \quad \text{(by Lemma~\ref{lem:exactone})} \\
=&\ \sum_{\substack{{\bf c'} \in \mcompn, \\ {\bf c} \opare {\bf c'} }} \sum_{\substack{ 1 \leq i_1 \leq i_2 \leq \cdots  \leq  i_{|{\bf c}|},  \\ \text{$i_k <i_{k+1}$ iff $k \in \Des({\bf c'})$}}} \prod_{j=1}^{|{\bf c}|} x_{g_{{\bf c'}}(j), {i_j}} \quad \text{(by Lemma~\ref{lem:equichar})}\\
=&\ \sum \limits_{{\bf c} \opare {\bf c}'} M_{{\bf c}'} \quad \text{(by Eq.~(\ref{eq:nw}))}.
\end{align*}
This completes the proof.
\end{proof}

Next we show that the fundamental multi-quasisymmetric functions form another ${\bf k}$-basis of $\mqsym$. To prove this result, we recall the following concepts and a result in~\mcite{GR14}.

\begin{defn}~\mcite{GR14}
Let $S$ be a set. An $S \times S$-{\bf matrix over ${\bf k}$} is a family $(a_{s,t})_{(s,t) \in S \times S} \in {\bf k}^{S \times S}$ of elements of ${\bf k}$ indexed by elements of $S \times S$.
\end{defn}

\begin{defn}~\mcite{GR14}  Let $(S, \leq_S)$ be a poset and $A=(a_{s,t})_{(s,t) \in S \times S} $ be an $S \times S$-matrix.
\begin{enumerate}
\item The matrix $A$ is called {\bf triangular} if $a_{s,t}=0$ for every $(s,t) \in S \times S$ which does not satisfy $t \leq_S s$;
\item The matrix $A$ is called {\bf invertibly triangular} if $A$ is triangular and $a_{s,s}$ is invertible for every $s \in S$.
\end{enumerate}
\end{defn}

\begin{defn}~\mcite{GR14}
Let $N$ be a ${\bf k}$-module and $S$ be a finite poset. Let $(e_s)_{s \in S}$ and $(f_s)_{s \in S}$ be two families of elements of $N$. We call the family $(e_s)_{s \in S}$ {\bf expands invertibly triangularly in the family} $(f_s)_{s \in S}$ if there is an invertibly triangular $S \times S$-matrix $A$ such that for each $s \in S$,
\begin{align*}
e_s= \sum \limits_{t \in S} a_{s,t} f_t.
\end{align*}
\end{defn}

\begin{lemma}\mcite{GR14}
Let $N$ be a ${\bf k}$-module and $S$ be a finite poset. Let $(e_s)_{s \in S}$ and $(f_s)_{s \in S}$ be two families of elements of $N$.  Assume that the family $(e_s)_{s \in S}$  expands invertibly triangularly in the family $(f_s)_{s \in S}$. Then the family $(e_s)_{s \in S}$ is a basis of the ${\bf k}$-module $N$ if and only if the family $(f_s)_{s \in S}$ is a basis of the ${\bf k}$-module $N$.
\mlabel{lem:basisequi}
\end{lemma}

\begin{prop}
The set $\{F_{\bf c} \mid {\bf c} \in \mcomp \}$ is a $\bf k$-basis of $\mqsym$.
\mlabel{pp:fbasis}
\end{prop}

\begin{proof}
By Proposition~\mref{prop:kbasis}, the set $\{M_{\bf w} \mid {\bf w} \in \mcomp \}$ is a $\bf k$-basis of $\mqsym$. Denote by $\mqsym_n$ the subspace spanned by $\{M_{\bf w} \mid {\bf w} \in \mcompn \}$. Then $\mqsym$ is a graded space $\mqsym=\mathop{\oplus}\limits_{n \in \mathbb{N}} \mqsym_n$ with $\mqsym_0=\bf k$. For each $n \in \mathbb{P}$, $(\mcompn, \opar)$ is a finite poset. By Proposition~\mref{prop:mcom}, the family $(F_{\bf c})_{{\bf c} \in \mcompn}$ expands invertibly triangularly in the family $(M_{\bf w})_{{\bf w} \in \mcompn}$. As $(M_{\bf w})_{{\bf w} \in \mcompn}$ is a ${\bf k}$-basis of $\mqsym_n$, by Lemma~\mref{lem:basisequi} $(F_{\bf c})_{{\bf c} \in \mcompn}$ is also a ${\bf k}$-basis of $\mqsym_n$. So $\{F_{\bf c} \mid {\bf c} \in \mcomp \}$ is a $\bf k$-basis of $\mqsym$.
\end{proof}

\subsection{Multi-compositions with semigroup exponents and quasi-shuffle product} \mlabel{ss:mqsymqs}
We now extend the multi-quasisymmetric functions to the ones with semigroup exponents in order to deal with weak compositions and Rota-Baxter algebras.

For the set $\sm$ and the corresponding sequence $Y=\{y_1,y_2, \ldots, y_m\}$ of variables, the monomials $y_{{i_1}}^{a_1}\cdots y_{{i_k}}^{a_k}$ in $\bfk[Y]$ can be viewed as the maps $f$ from $\sm$ to $\mathbb{N}$ such that $f({i_j})=a_j, 1\leq j\leq k$, and $f(i)=0$ otherwise. Replacing $\mathbb{N}$ by a commutative additive monoid $\E$ leads to (formal) monomials with exponents in $\E$ and further to quasisymmetric functions
with semigroup exponents~\mcite{YGT19}.

To be precise, let $\E$ be a commutative additive monoid with zero $0$ such that $\E \setminus \{0\}$ is a subsemigroup. Define
the set of maps
\begin{align*}
\ea :=\big\{ f : \A \rightarrow \E \big\}.
\end{align*}
The addition on $\E$ equips $\ea$ with an additive monoid structure under the addition
\begin{equation*}
(f+g)(i):=f(i)+g(i) \,\, \text{for $f,g \in \ea$ and $i \in \A$.}
\mlabel{eq:fsumg}
\end{equation*}

\begin{defn}
A vector ${\bf w}=(w_1, \ldots, w_k)\in (\ea)^k$ is called a {\bf multi-composition with semigroup exponent} $\E$ if none of $w_i$ is the zero map. The {\bf length} $\ell({\bf w})$ of ${\bf w}$ is $k$.
\end{defn}

By convention, we denote the trivial multi-composition with semigroup exponent of length 0 by $\etree$.  Denote the set of all multi-compositions with semigroup exponent by $\sbx$.
We next equip the free $\bfk$-module $\bfk \sbx$ with an algebra structure by putting it in the more general context of quasi-shuffle algebras.

Let $(A, \cdot)$ be a commutative algebra and define the \bfk-module $T(A) := \bigoplus_{n \geq 0} A^{\ot n}$.
We write a pure tensor ${\bf a}=a_1 \ot \cdots \ot a_n \in A^{\ot n}$.
For $a \in A$ and ${\bf a}=a_1\ot \cdots\ot a_n \in T(A)$, abbreviate $a\ot {\bf a}$ for $a\ot a_1\ot \cdots\ot a_n$.
Define recursively the quasi-shuffle product $\ast:T(A) \ot T(A) \to T(A)$ by
\begin{enumerate}
\item $1_{\bfk} \ast {\bf a} :={\bf a} =:{\bf a} \ast 1_{\bfk} \,\text{ for } {\bf a} \in T(A).$
\item For $a\ot {\bf a}, b\ot {\bf b} \in T(A)$,
\begin{align}
(a\ot {\bf a}) \ast (b\ot {\bf b}):=a\ot ({\bf a}\ast (b\ot {\bf b}))+b\ot \big((a\ot {\bf a}) \ast {\bf b}\big)+(a \cdot b)\ot ({\bf a}\ast {\bf b}).
\mlabel{eq:qshuffle}
\end{align}
\end{enumerate}
Recall~\mcite{G12,GX13,H00} that $(T(A),\ast)$ is a commutative algebra with identity $1_{\bf k}$, called the {\bf quasi-shuffle algebra} and denoted by $\mathrm{QS}(A)$.

Now take $A$ to be the semigroup algebra $\bfk\ea$. We have the linear bijection from the free module $\bfk \sbx$ to $T(\bfk \ea)$ by sending ${\bf w}=(w_1,\ldots,w_n)\in \sbx$ to $w_1\ot \cdots \ot w_n\in (\bfk [m]^E)^{\ot n}$. Then by transporting of structures, the quasi-shuffle product $\ast$ on $T(\bfk \ea)$ gives a product, still denoted by $\ast$, on $\bfk \sbx$. Then we have

\begin{prop}
The triple $(\bfk \sbx, \ast, \etree)$ is a commutative algebra.
\mlabel{prop:calg}
\end{prop}

Note that, the recursive formula of the product $\ast$ for $\sbx$ is obtained by simply replacing the tensor product in Eq.~\meqref{eq:qshuffle} by the Cartesian product.

\subsection{Multi-quasisymmetric functions with semigroup exponents}
Now we are in the position to define multi-quasisymmetric functions with semigroup exponents.

A semigroup $\E$ is called {\bf additively finite} if for any $\e \in \E$, there are only finite number of pairs $(\e_1,\e_2) \in \E^2$ such that $\e_1+\e_2=\e$~\mcite{YGT19}.

\noindent
{\bf Notation. } {\it For the rest of the paper}, we always assume that $\E$ is a commutative additive finite monoid with zero element 0 such that $\E \setminus \{0\}$ is a subsemigroup, unless otherwise specified. Alternatively, $E$ is obtained by adding a zero element to a semigroup.

To obtain a monomial form of $f \in \ea$ as for the classical monomials, we identify $f$ with its locus $\{(\ax,f(\ax)) \mid \ax \in \mathrm{supp}(f) \}$ which is expressed as the formal product
\begin{align}
w(f) :=\prod \limits_{\ax \in \mathrm{supp}(f)}\ax^{f(\ax)} =\ax_{1}^{\e_1} \cdots \ax_{j}^{\e_j},
\mlabel{eq:formal}
\end{align}
where each $\ax_{k}\in \mathrm{supp}(f)$ and $\e_k= f(\ax_{k}) \in \E$ with $\e_k \neq 0$. Denote $w(f)=\etree$ for $f=0$.
The product of the formal monomials $w(f)$ and $w(g)$  corresponds to the map $f+g$:
\begin{equation*}
	w(f) \cdot w(g) = w(f+g).
\mlabel{eq:fprod}
\end{equation*}
From now on, we identify $f\in \ea$ with the formal monomial $w=w(f)=\ax_{1}^{\e_1} \cdots \ax_{j}^{\e_j}.$

Then one can also express a multi-composition with semigroup exponents ${\bf w}=(w_1,\ldots,w_k)$ in the matrix form as in Eq.~\meqref{eq:multcomp}, simply taking $a_{i,j}$ to be in $\E$ instead of in $\NN$.

\begin{remark}
If $m=1$, then the $\mcswse$ are identified with $\E$-compositions introduced in~\mcite{YGT19} via the bijection $1^{\e} \leftrightarrow \e$. If $\E=\mathbb{N}$, then the multi-compositions with semigroup exponents are identified with the multi-compositions in Definition~\ref{def:mc} .
\mlabel{rem:identify}
\end{remark}

We still use the set of variables
$$\X:={\bf x}_{\sm}:=\{x_{i,j}\,|\, i \in \sm,j\in \PP\}$$
in Eq.~(\ref{eq:eks}). Let $\bfk [[\X]]^{\E}$ denote the set of possibly infinite $\bfk$-linear combinations of $\X^\E$ with multiplication induced by $\X^\E$, called the {\bf formal power series algebra} in $\X$ with semigroup exponents in $\E$.

\begin{remark}
The additive finiteness condition of $\E$ is to ensure that the multiplication in $\bfk [[\X]]^{\E}$ is well defined.
\end{remark}

Given an element $w = \ax_{1}^{\e_1} \cdots \ax_{k}^{\e_k}\in \ea$  and a positive integer $j$,
denote by $\x^{w}_{j}$  the monomial
\begin{equation}
\x^{w}_{j} := x_{\ax_{1},j}^{\e_1}  \cdots x_{\ax_{k},j}^{\e_k}.
\mlabel{eq:ekswj}
\end{equation}
Then any monomial in $\bfk [[\X]]^{\E}$ can be uniquely written as
\begin{equation*}
\x^{\bf w}_{{\bf j}} := \x^{w_1}_{j_1}   \cdots \x^{w_n}_{j_n},
\end{equation*}
for some ${\bf w}=(w_1, \ldots, w_n)\in \sbx$ and positive integers ${\bf j}=\{j_1 < \cdots < j_n \}$.

\begin{exam}
Let ${\bf w}=(w_1,w_2,w_3)=(\ax_{1}^{\e_1}\ax_{2}^{\e_2}, \ax_{1}^{\e_3}, \ax_{2}^{\e_4}\ax_{3}^{\e_5}) \in \sbx$ and ${\bf j}=\{ j_1 < j_2 <j_3\}$. Then
\begin{align*}
\x^{w_1}_{j}=\x^{\e_1}_{\ax_1,j}\x^{\e_2}_{\ax_{2},j} \,\, \text{ and } \,\, \x^{{\bf w}}_{\bf j}= \x^{\e_1}_{\ax_1,j_1}\x^{\e_2}_{\ax_{2},j_1} \x^{\e_3}_{\ax_1,j_2} \x^{\e_4}_{\ax_2,j_3} \x^{\e_5}_{\ax_3,j_3}.
\end{align*}
\end{exam}

As in the case of the classical quasisymmetric functions, an element $f \in \bfk [[\X]]^{\E}$ is called {\bf a multi-quasisymmetric function with semigroup exponents} if for each $(w_1,\ldots, w_n) \in \sbx$, the coefficients of monomials $\x^{w_1}_{j_1} \cdots \x^{w_n}_{j_n}$ in $f$ are equal for all positive integers $1 \leq j_1 < \cdots <j_n $.  We denote the set of all multi-quasisymmetric functions with semigroup exponents by $\qsym$.

\begin{defn}
Let ${\bf w}=(w_1, \ldots, w_n) \in \sbx$ be a $\mcwse$. Define {\bf the monomial multi-quasisymmetric function with semigroup exponents} by
\begin{align}
M_{{\bf w}} := \sum \limits_{{\bf j}=\{ j_1 < \cdots < j_n \} } x^{\bf w}_{{\bf j}} \in \bfk [[X]]^{\E},
\mlabel{eq:mi}
\end{align}
where ${\bf j}$ runs through all sets of $n$ distinct positive integers. In particular, $M_{\etree}=1_\bfk$.
\label{def:monoq}
\end{defn}

Next we show that the subspace
$$\qsym:=\bfk\,\{M_{\bf w}\,|\, {\bf w}\in \sbx\} $$
of $\bfk [[\X]]^{\E}$  is a subalgebra.

\begin{theorem}
The linear map
\begin{align}
M: \bfk \sbx \rightarrow \qsym,\quad {\bf w} \mapsto M_{{\bf w}}
\label{eq:defm}
\end{align}
is bijective and preserves the multiplications, that is,
\begin{align}
M({\bf w}_1 \ast {\bf w}_2)= M({\bf w}_1)M({\bf w}_2)= M_{{\bf w}_1}M_{{\bf w}_2} \,\, \text{for any ${\bf w}_1,{\bf w}_2 \in \sbx$}.
\mlabel{eq:isomorphism}
\end{align}
\mlabel{thm:isomorphism}
\end{theorem}

\begin{proof}
Similar to the proof of Lemma~2.3 in~\cite{YGT19}, we have that $\{M_{\bf w} \mid  {\bf w} \in \sbx\}$ is a ${\bfk}$-basis for $\qsym$. Hence the map $M$ given by Eq.~(\ref{eq:defm}) is a bijection.
Suppose
\begin{align*}
{\bf w}_1=(w_1, \ldots, w_{n_1})\,\text{ and }\, {\bf w}_2=(u_1,  \ldots, u_{n_2}).
\end{align*}
We prove Eq.~(\mref{eq:isomorphism}) by induction on $n_1+n_2\geq 0$.
If $n_1+n_2=0$, then $n_1=n_2=0$ and ${\bf w}_1={\bf w}_2=\etree$. So $M_{{\bf w}_1}=M_{{\bf w}_2}=1_\bfk$ and Eq.~(\mref{eq:isomorphism}) is valid.
Assume that Eq.~(\ref{eq:isomorphism}) holds for $n_1+n_2 \leq \ell$ for an integer $\ell \geq 0$, and consider the case $n_1+n_2=\ell+1$.
We have
\begin{align*}
M_{ {\bf w}_1 } M_{ {\bf w}_2 }=&\  \sum_{\substack{1 \leq i_1 < \cdots < i_{n_1}  \\  1 \leq j_1 < \cdots < j_{n_2} }} x^{w_1}_{i_1} \cdots x^{w_{n_1}}_{i_{n_1}} x^{u_1}_{j_1} \cdots x^{u_{n_2}}_{j_{n_2}} \,\,\quad \quad\quad (\text{by Eq.~(\mref{eq:mi})})\\
=&\ \sum_{\substack{1 \leq i_1 < \cdots < i_{n_1}  \\  1 \leq j_1 < \cdots < j_{n_2}  \\ i_1< j_1}} x^{w_1}_{i_1} \cdots x^{w_{n_1}}_{i_{n_1}} x^{u_1}_{j_1} \cdots x^{u_{n_2}}_{j_{n_2}}+\sum_{\substack{1 \leq i_1 < \cdots < i_{n_1} \\  1 \leq j_1 < \cdots < j_{n_2} \\ j_1< i_1}} x^{w_1}_{i_1} \cdots x^{w_{n_1}}_{i_{n_1}} x^{u_1}_{j_1} \cdots x^{u_{n_2}}_{j_{n_2}}\\
&\ +\sum_{\substack{1 \leq i_1 < \cdots < i_{n_1} \\  1 \leq j_1 < \cdots < j_{n_2}  \\ i_1=j_1}} x^{w_1}_{i_1} \cdots x^{w_{n_1}}_{i_{n_1}} x^{u_1}_{j_1} \cdots x^{u_{n_2}}_{j_{n_2}} \,\, \quad \quad\quad(\text{by comparing $i_1$ and $j_1$})\\
=&\ \sum \limits_{1 \leq i_1} x_{i_1}^{w_1} \sum_{\substack{ i_1  < \cdots < i_{n_1}  \\  i_1 < j_1  < \cdots < j_{n_2} }}  x^{w_2}_{i_2} \cdots x^{w_{n_1}}_{i_{n_1}} x^{u_1}_{j_1}  \cdots x^{u_{n_2}}_{j_{n_2}}\\
&\ + \sum \limits_{1 \leq j_1} x_{j_1}^{u_1} \sum_{\substack{ j_1 <i_1< \cdots < i_{n_1} \\  j_1  < \cdots < j_{n_2} }} x^{w_1}_{i_1} \cdots x^{w_{n_1}}_{i_{n_1}} x^{u_2}_{j_2} \cdots x^{u_{n_2}}_{j_{n_2}}\\
&\ +\sum \limits_{1 \leq i_1=j_1} x_{i_1}^{w_1} x_{j_1}^{u_1} \sum_{\substack{ i_1  < \cdots < i_{n_1}  \\  j_1 < \cdots < j_{n_2} }}  x^{w_2}_{i_2} \cdots x^{w_{n_1}}_{i_{n_1}} x^{u_2}_{j_2} \cdots x^{u_{n_2}}_{j_{n_2}}\\
&\ \hspace{3.5cm} \quad (\text{by the commutativity of variables})\\
=&\ M_{(w_1,(w_2, \ldots, w_{n_1}) \ast {\bf w}_2)}+M_{(u_1, {\bf w}_1 \ast (u_2, \ldots, u_{n_2}))}+M_{(w_1 \cdot u_1, (w_2, \ldots, w_{n_1}) \ast (u_2, \ldots, u_{n_2}))}\\
 &\ \hspace{4cm} \text{(by the induction hypothesis)}\\
=&\ M_{{\bf w}_1 \ast {\bf w}_2},
\end{align*}
as required.
\end{proof}

As an immediate consequence of Proposition~\mref{prop:calg} and Theorem~\mref{thm:isomorphism}, we obtain

\begin{coro}
The submodule $\qsym$ is a subalgebra of $\bfk [[\X]]^{\E}$ with a $\bfk$-basis
\begin{align*}
\qsym= \bfk \{M_{{\bf w}} \mid  {\bf w} \in \sbx\}.
\end{align*}
\mlabel{co:eqsymsubalg}
\end{coro}
\vskip-0.3in

We have the following special cases.
\begin{remark}
\begin{enumerate}
\item
Let $m=1$. Then $\qsym$ is the $\E$-quasisymmetric functions introduced
in~\mcite{YGT19}.
\item
Let $\E=\mathbb{N}$. Then $\sbx$ is the set of multi-compositions, and $\qsym$ recovers the space $\mqsym$ of multi-quasisymmetric functions introduced in Section~\mref{ss:mqsym}. Hence the first statement of Theorem~\mref{pp:mqsym} follows from Corollary~\mref{co:eqsymsubalg}.
\end{enumerate}
\mlabel{rk:ecases}
\end{remark}

We consider another case for later applications.
First recall the commutative additive finite monoid $\wmn$ given in~\mcite{YGT19}.
Let $\varepsilon$ be a symbol and $\wmn:= \mathbb{N} \sqcup \{ \varepsilon \}$. The addition of $\wmn$ is the usual addition of $\mathbb{N}$ together with
$$0+ \varepsilon=\varepsilon+0=\varepsilon+\varepsilon=\varepsilon\,\text{ and }\, n+ \varepsilon=\varepsilon+n=n$$ for all $n \geq 1$.
\delete{
\song{Here is modified}For $w \in [m]^{\wmn}$, we call $w$ {\bf weak} if viewed as a map from $[m]$ to $\wmn$, $w(i) \neq 0$ for all $1 \leq i \leq m$. \li{This means that $w$ is in $[m]^{\widetilde{\PP}}$ for $\widetilde{\PP}=\PP\sqcup \{\varepsilon\}$?}
}
Denote $\widetilde{\PP}:=\PP \sqcup \{\varepsilon\}$ and consider a subset of $\snx$ defined by
\begin{align*}
W\snx:=\{ (w_1,\ldots, w_n) \in \snx \mid w_i\in [m]^{\widetilde{\PP}}, 1 \leq i \leq n \}\cup \{\etree\}.
\end{align*}
Thus the free $\bfk$-submodule ${\bf k} W\snx$ is isomorphic to $T(\bfk [m]^{\widetilde{\PP}})$ for the commutative algebra $\bfk [m]^{\widetilde{\PP}}$.
Then by Eq.~(\ref{eq:qshuffle}) with $A=\bfk [m]^{\widetilde{\PP}}$, we see that ${\bf k} W\snx$ is a quasi-shuffle subalgebra of $({\bf k} \snx, \ast)$. By Theorem~\ref{thm:isomorphism}, ${\bf k} \snx \cong \mqsymn$ and hence ${\bf k} W\snx$ corresponds to a subalgebra of $\mqsymn$. We denote this subalgebra of $\mqsymn$ by $\wmqsym$ and call it {\bf the algebra of weak multi-quasisymmetric functions.} By Corollary~\ref{co:eqsymsubalg}, $\mqsymn= \bfk \{M_{{\bf w}} \mid  {\bf w} \in \snx\}$ and hence
\begin{align*}
\wmqsym=\bfk \{M_{{\bf w }} \mid {\bf w} \in W\snx \}.
\end{align*}


\section{Hopf algebraic structure of multi-quasisymmetric functions} \mlabel{secb}
In this section, we obtain the Hopf algebraic structure of a general quasi-shuffle algebra and apply it to multi-quasisymmetric functions with semigroup exponents and in particular to weak multi-quasisymmetric functions.

\subsection{Hopf algebraic structure of quasi-shuffle algebras}
For a commutative algebra $A$ whose product is denoted by $\cdot$, consider the quasi-shuffle algebra $(\mathrm{QS}(A),\ast)$ defined in Eq.~\meqref{eq:qshuffle}.
For ${\bf a}=a_1 \ot \cdots \ot a_n \in \mathrm{QS}(A)$, define the length $\ell({\bf a})$ of ${\bf a}$ to be $n$.
For ${\bf a}=a_1 \ot \cdots \ot a_{n_1}, {\bf b}=b_1 \ot \cdots \ot  b_{n_2} \in \mathrm{QS}(A)$, define the concatenation product ${\bf a} {\bf b} :=a_1 \ot \cdots \ot a_{n_1} \ot b_1 \ot \cdots \ot b_{n_2}$. Define the deconcatenation coproduct $\Delta$ and the counit $\epsilon$ on  $\qsa$ by
\begin{align}
\Delta({\bf a}) :=&\ \sum \limits_{i=0}^n  (a_1 \ot \cdots \ot a_i) \oot (a_{i+1} \ot \cdots \ot a_n), \mlabel{eq:del}\\
\epsilon({\bf a}) :=&\ \delta_{{\bf a}, 1_{\bfk}}, \mlabel{eq:cou}
\end{align}
for ${\bf a}=a_1 \ot \cdots \ot a_n \in \qsa$.
By Eqs.~(\mref{eq:del}) and~(\mref{eq:cou}), $\Delta$ is coassociative and $\epsilon$ is a counit. In analogous to~\cite[Theorem~3.1]{H00}, we have the following result.

\begin{theorem}
$(\qsa, \ast,1_{\bf k}, \Delta, \epsilon)$ is a bialgebra.
\mlabel{thm:bialgebra}
\end{theorem}

\begin{proof}
The proof follows the same idea as for the classical case of~\cite[Theorem~3.1]{H00}, when the commutative algebra $A$ is a type of semigroup algebra $\bfk Z$ for locally finite graded semigroup $Z$. Since we consider a commutative algebra with no restrictions, we provide the full detail.

By Eq.~(\mref{eq:cou}), $\epsilon$ is an algebra homomorphism with respect to $\ast$. It remains to verify
\begin{align}
\Delta({\bf a}) \ast \Delta({\bf b})=\Delta ({\bf a} \ast {\bf b})
\mlabel{eq:mor}
\end{align}
for monomials ${\bf a}, {\bf b} \in \qsa$.

We prove Eq.~(\mref{eq:mor}) by induction on $\ell({\bf a})+ \ell({\bf b})$. If ${\bf a} =1_{\bf k}$ or ${\bf b} =1_{\bf k}$, then Eq.~(\ref{eq:mor}) holds directly. For the induction step, assume ${\bf a}=a \ot {\bf a}'$ and ${\bf b}=b \ot {\bf b}'$. Adopting Sweedler's notation~\mcite{S69}, write
\begin{align*}
\Delta({\bf a}')=\sum_{({\bf a}')} {\bf a}'_{(1)} \oot {\bf a}'_{(2)} \, \text{ and }\, \Delta({\bf b}')=\sum_{({\bf b}')} {\bf b}'_{(1)} \oot {\bf b}'_{(2)}.
\end{align*}
Then
\begin{align*}
&\ \Delta({\bf a}) \ast \Delta( {\bf b})\\
=&\ \left(\sum_{({\bf a}')} (a \ot {\bf a}'_{(1)}) \oot {\bf a}'_{(2)} + 1_{\bfk} \oot (a \ot {\bf a}')\right) \ast \left(\sum_{({\bf b}')} (b \ot {\bf b}'_{(1)}) \oot {\bf b}'_{(2)}+ 1_{\bfk} \oot (b \ot {\bf b}')\right) \,\,\quad\quad\quad \text{(by Eq.~(\ref{eq:del}))}\\
=&\ \sum_{({\bf a}'), ({\bf b}')} \big((a \ot {\bf a}'_{(1)}) \ast (b \ot {\bf b}'_{(1)}) \big) \oot ({\bf a}'_{(2)} \ast {\bf b}'_{(2)})+ \sum_{({\bf a}')} (a \ot {\bf a}'_{(1)}) \oot \big({\bf a}'_{(2)} \ast (b \ot {\bf b}')\big)\\
&\ +\sum_{({\bf b}')} (b \ot {\bf b}'_{(1)}) \oot \big((a \ot {\bf a}') \ast {\bf b}'_{(2)} \big)+ 1_{\bf k} \oot \big((a \ot {\bf a}') \ast (b,{\bf b}') \big)\\
=&\ \sum_{({\bf a}'), ({\bf b}')} \big(a \ot ({\bf a}'_{(1)} \ast (b \ot {\bf b}'_{(1)}))\big) \oot ({\bf a}'_{(2)} \ast {\bf b}'_{(2)})+ \sum_{({\bf a}'), ({\bf b}')} \big(b \ot ((a \ot {\bf a}'_{(1)}) \ast {\bf b}'_{(1)}) \big) \oot ({\bf a}'_{(2)} \ast {\bf b}'_{(2)})\\
&\ +\sum_{({\bf a}'), ({\bf b}')}  ( (a \cdot b) \ot ({\bf a}'_{(1)} \ast {\bf b}'_{(1)})) \oot ({\bf a}'_{(2)} \ast {\bf b}'_{(2)})+ \sum_{({\bf a}')}  (a \ot {\bf a}'_{(1)}) \oot \big({\bf a}'_{(2)} \ast (b \ot {\bf b})\big)\\
&\ +\sum_{({\bf b}')}  (b \ot  {\bf b}'_{(1)}) \oot \big((a \ot {\bf a}') \ast {\bf b}'_{(2)} \big)+ 1_{\bf k} \oot (a \ot ({\bf a}' \ast {\bf b}))+ 1_{\bf k} \oot (b \ot  ({\bf a} \ast {\bf b}'))+ 1_{\bf k} \oot ((a \cdot b) \ot ({\bf a}' \ast {\bf b}'))\\
&\ \hspace{12pc} \,\, \text{(by Eq.~(\ref{eq:qshuffle}))}\\
=&\ \sum_{({\bf a}'), ({\bf b}')} \big(a \ot ({\bf a}'_{(1)} \ast (b \ot {\bf b}'_{(1)}))\big) \oot ({\bf a}'_{(2)} \ast {\bf b}'_{(2)})+ \sum_{({\bf a}')}  (a \ot {\bf a}'_{(1)}) \oot \big({\bf a}'_{(2)} \ast (b \ot {\bf b})\big)+ 1_{\bf k} \oot (a \ot ({\bf a}' \ast {\bf b}))\\
&\ +\sum_{({\bf a}'), ({\bf b}')} \big(b \ot ((a \ot {\bf a}'_{(1)}) \ast {\bf b}'_{(1)}) \big) \oot ({\bf a}'_{(2)} \ast {\bf b}'_{(2)})+\sum_{({\bf b}')}  (b \ot  {\bf b}'_{(1)}) \oot \big((a \ot {\bf a}') \ast {\bf b}'_{(2)} \big)+1_{\bf k} \oot (b \ot  ({\bf a} \ast {\bf b}'))\\
&\ +\sum_{({\bf a}'), ({\bf b}')}  ( (a \cdot b) \ot ({\bf a}'_{(1)} \ast {\bf b}'_{(1)})) \oot ({\bf a}'_{(2)} \ast {\bf b}'_{(2)})+1_{\bf k} \oot ((a \cdot b) \ot ({\bf a}' \ast {\bf b}'))\\
=&\ \sum_{({\bf a}' \ast {\bf b})} \big( a\ot ({\bf a}' \ast {\bf b})_{(1)} \big) \oot ({\bf a}' \ast {\bf b})_{(2)}+  1_{\bf k} \oot (a \ot ({\bf a}' \ast {\bf b}))\\
&\ + \sum_{({\bf a} \ast {\bf b}')} \big( b \ot ({\bf a} \ast {\bf b}')_{(1)} \big) \oot ({\bf a} \ast {\bf b}')_{(2)}+ 1_{\bf k} \oot (b \ot ( {\bf a} \ast {\bf b}'))\\
&\ + \sum_{({\bf a}' \ast {\bf b}')} \big( (a \cdot b) \ot ({\bf a}' \ast {\bf b}')_{(1)}) \oot ({\bf a}' \ast {\bf b}')_{(2)}+ 1_{\bf k} \oot ((a \cdot b) \ot ({\bf a}' \ast {\bf b}'))
\hspace{0.5cm}\text{(by the induction hypothesis)}\\
=&\ \Delta \big((a \ot ({\bf a}' \ast {\bf b}))+(b \ot ({\bf a} \ast {\bf b}'))+((a\cdot b) \ot ({\bf a}' \ast {\bf b}')) \big) \\
=&\  \Delta ({\bf a} \ast {\bf b}).
\end{align*}
This completes the proof.
\end{proof}

\begin{defn}
Let ${\bf a}=a_1 \ot  \cdots \ot a_n \in \qsa$ with $\ell({\bf a})=n$ and $L=(\ell_1, \ldots, \ell_k)$ be a composition of $n$.
\begin{enumerate}
\item The {\bf reversal} ${\bf a}^r$ of ${\bf a}$ is defined to be $a_n \ot \cdots \ot a_1$.

\item Define $L \circ {\bf a}$ to be
\begin{align*}
L \circ {\bf a} =(a_{1}\cdot a_{2} \cdot \, \cdots \, \cdot a_{\ell_1}) \ot (a_{\ell_1+1}\cdot a_{\ell_1+2}\cdot \, \cdots \, \cdot a_{\ell_1+\ell_2}) \ot  \cdots \ot (a_ {\ell_1+ \cdots +\ell_{k-1}+1}\cdot  a_{\ell_1+ \cdots +\ell_{k-1}+2}\cdot \, \cdots \, \cdot a_{n} ),
\end{align*}
where $\cdot$ is the product in the algebra $A$.
\end{enumerate}
\mlabel{def:reversal}
\end{defn}

Thus $L \circ {\bf a}$ can be obtained in two steps:

{\bf Step 1:} Divide ${\bf a}$ into $n$ parts from left to right such that the $i$-th part has $\ell_i$ elements of $A$:
\begin{align*}
a_1 \ot \cdots \ot a_{\ell_1} \, \| \,  a_{\ell_1+1} \ot \cdots \ot a_{\ell_1+\ell_2} \,  \| \,  \ldots \, \| \, a_{\ell_1+\cdots+ \ell_{k-1}+1} \ot \cdots \ot a_n.
\end{align*}

{\bf Step 2:} For each part, multiply the $l_i$ elements of this part using the product $\cdot$ of $A$:
\begin{align*}
a_{1}\cdot  \ldots  \cdot a_{\ell_1} \, \| \, a_{\ell_1+1}\cdot \ldots \cdot a_{\ell_1+\ell_2} \, \| \, \ldots \, \| \, a_{\ell_1+\cdots+\ell_{k-1}+1} \cdot \ldots \cdot a_{n}.
\end{align*}

Then we obtain
\begin{theorem}
The bialgebra $(\qsa, \ast,1_{\bf k}, \Delta, \epsilon)$ is a Hopf algebra with the antipode $S$ given by
\begin{align}
S({\bf a}):=(-1)^{\ell({\bf a})} \sum \limits_{L \models \ell({\bf a})} L \circ {\bf a}^r \,\,\,\, \text{for ${\bf a} \in \qsa$}.
\mlabel{eq:antipode}
\end{align}
Here we use the convention that $\sum \limits_{L \models \ell(1_{\bf k})} L \circ 1_{\bf k}=1_{\bf k}$.
\mlabel{thm:hopf}
\end{theorem}

\begin{proof}
The argument follows the proof of Theorem 3.2 in~\mcite{H00}.
Since $\qsa$ is a bialgebra, we only need to show that $S$ is the antipode of $\qsa$. For ${\bf a} \in \qsa$,  write
$$\Delta({\bf a})= \sum \limits_{({\bf a})} {\bf a}_{(1)} \oot {\bf a}_{(2)}.$$
It suffices to show that the linear map $S:\qsa \rightarrow \qsa$ satisfies
\begin{align*}
\sum \limits_{({\bf a})} S({\bf a}_{(1)}) \ast {\bf a}_{(2)}=\epsilon({\bf a}) 1_{\bf k}=\sum \limits_{({\bf a})} {\bf a}_{(1)} \ast S({\bf a}_{(2)}).
\end{align*}
We just prove the left equation since the right one can be proved in the same way.
If ${\bf a} = 1_{\bfk}$, then
$$\Delta({\bf a}) =1_{\bfk} \oot 1_{\bfk},\, \epsilon({\bf a}) = 1_{\bfk}$$
and $S(1_{\bfk}) = 1_{\bfk}$ by Eq.~(\mref{eq:antipode}). So the left equation holds.

Consider ${\bf a} \neq 1_{\bfk}$. Since $\epsilon({\bf a}) = 0$ by Eq.~(\ref{eq:cou}), it is enough to show
\begin{align}
\sum \limits_{({\bf a})} S({\bf a}_{(1)}) \ast {\bf a}_{(2)}=0.
\mlabel{eq:antipode1}
\end{align}
We prove Eq.~(\mref{eq:antipode1}) by induction on the length $\ell({\bf a})\geq 1$. If $\ell({\bf a}) =1$, then by Eq.~(\ref{eq:del}),
$$\Delta({\bf a}) = 1_{\bfk} \oot {\bf a} + {\bf a}\oot 1_{\bfk}.$$
Then Eq.~(\mref{eq:antipode}) means
$$S({\bf a}) =-{\bf a}.$$
Thus
\[
\sum \limits_{({\bf a})} S({\bf a}_{(1)}) \ast {\bf a}_{(2)} = S(1_{\bfk}) \ast {\bf a} + S({\bf a}) \ast 1_{\bfk} = {\bf a} -  {\bf a} = 0,
\]
and so Eq.~(\mref{eq:antipode1}) holds. Suppose that Eq.~(\mref{eq:antipode1}) holds for $\ell({\bf a})<n$ for an integer $n \geq 1$, and consider the case of $\ell({\bf a})=n$.
Write ${\bf a}=a_1 \ot \cdots \ot a_n$. Then
\begin{align*}
&\ \sum \limits_{{\bf a}} S({\bf a}_{(1)}) \ast {\bf a}_{(2)} \\
=& \sum \limits_{i=0}^n S (a_1 \ot  \cdots \ot a_i )\ast  (a_{i+1} \ot \cdots \ot a_n) \quad \quad (\text{by Eq.~(\mref{eq:del})})\\
=&\ \sum \limits_{i=0}^n (-1)^i  \left( \sum \limits_{L \models i} L \circ (a_i \ot \cdots \ot a_1) \right) \ast  (a_{i+1} \ot \cdots \ot a_n)\,\, \quad\text{(by Eq.~(\mref{eq:antipode}))}.
\end{align*}
So it is enough to verify
\begin{align}
(-1)^n \sum \limits_{L \models n} L \circ {\bf a}^r =\sum \limits_{i=0}^{n-1} (-1)^{i+1} \left (\sum \limits_{L \models i} L \circ (a_{i} \ot a_{i-1} \ot \cdots \ot a_1)\right) \ast (a_{i+1} \ot \cdots \ot a_{n}).
\mlabel{eq:antipode2}
\end{align}
Now we consider the right hand of Eq.~(\mref{eq:antipode2}). If $i=0$, then
\begin{align*}
\left( \sum\limits_{L \models i} L \circ(a_i \ot a_{i-1} \ot \cdots \ot a_1)\right) \ast (a_{i+1} \ot \cdots \ot a_n)=1_{\bfk} \ast (a_1 \ot \cdots \ot a_n) =a_1 \ot \cdots \ot a_n.
\end{align*}
If $i \not=0$, then the first component of $L \circ (a_{i} \ot a_{i-1} \ot \cdots \ot a_1)$ is of the form $a_j \cdot a_{j+1} \cdot \cdots \cdot a_i$ for some $1 \leq j \leq i$. Hence the first component of the terms in the expansion $$\left( L \circ (a_{i} \ot a_{i-1} \ot \cdots \ot a_1)\right) \ast (a_{i+1} \ot \cdots \ot a_{n})$$
is one of the following three forms:
\begin{enumerate}
\item $a_j \cdot a_{j+1} \cdot \cdots \cdot a_i$, \,\, $1 \leq j \leq i \leq n-1$;

\item $a_j \cdot a_{j+1} \cdot \cdots \cdot a_i \cdot a_{i+1}$, \,\, $1 \leq j \leq i \leq n-1$;

\item $a_{i+1}$, \,\, $1 \leq j \leq i \leq n-1$.
\end{enumerate}
We will say that the term is of type $A_i$ in the first case, and of type $B_i$ in the latter two cases.
For a term that appears on the right hand of Eq.~(\mref{eq:antipode2}), if it is of type $A_k$ with $1 \leq k \leq n- 1$, then it will also occur as type $B_{k-1}$. The two occurrences will have opposite signs and hence will cancel each other. Thus the only terms that do not get canceled are those of type $B_n$, which will appear only for $i=n-1$ and have the coefficient $(-1)^n$. This gives the left hand side of Eq.~(\mref{eq:antipode2}), as required.
\end{proof}

\subsection{Hopf algebraic structures of $\qsym$ and $\wmqsym$}

Since $\A^{\E}$ is a commutative monoid and $\bfk \sbx$ is identified with the tensor algebra $T(\bfk \A^{\E})$, as a special case of Theorem~\mref{thm:hopf}, there is a Hopf algebraic structure on $\bfk \sbx$. Then by Theorem~\mref{thm:isomorphism}, and Eqs.~(\mref{eq:del}) and (\mref{eq:cou}), we obtain
\begin{prop}
The algebra $\qsym$ is a Hopf algebra, with coproduct $\Delta$, counit $\epsilon$ and antipode $S$ defined by
\begin{align*}
\Delta(M_{\bf w}) :=&\ \sum \limits_{{\bf w}={\bf w}_1 {\bf w}_2} M_{{\bf w}_1} \ot M_{{\bf w}_2}= \sum \limits_{i=0}^n M_{(w_1,\cdots, w_i)} \ot M_{(w_{i+1}, \cdots, w_n)},  \\
\epsilon(M_{\bf w}) :=&\ \delta_{{\bf w}, \etree}, \quad
S(M_{\bf w}):=(-1)^{\ell({\bf w})} \sum \limits_{L \models \ell({\bf w})} M_{L \circ {\bf w}^r}.
\end{align*}
for ${\bf w} \in \sbx$.
\mlabel{prop:hopf1}
\end{prop}

Taking $\E=\NN$, we obtain the second statement in Theorem~\mref{pp:mqsym}. Taking $\E=\widetilde{\mathbb{N}}$, we obtain $\mqsymn$ is a Hopf algebra. In particular, for
\begin{align*}
\wmqsym=\bfk \{M_{{\bf w }} \mid {\bf w} \in W\snx \},
\end{align*}
by Proposition~\ref{prop:hopf1}, we see that $\wmqsym$ is closed under the coproduct $\Delta$, the counit $\epsilon$ and the antipode $S$ of $\mqsymn$. Hence we obtain

\begin{coro}
The algebra $\wmqsym$ of weak multi-quasisymmetirc functions is a Hopf subalgebra of $\mqsymn$.
\mlabel{cor:subhopf}
\end{coro}

\section{Multi-quasisymmetric functions and free Rota-Baxter algebras}\mlabel{secd}
In this section, we prove that a scalar extension of the algebra $\wmqsym$ of weak multi-quasisymmetric functions in $m$ sequences of variables is isomorphic to the free commutative Rota-Baxter algebra on a finite set $Y$ of cardinality $m$.

\subsection{Quasi-shuffle algebras and mixable shuffle algebras}

Let $A$ be a commutative \bfk-algebra. The notion of mixable shuffles was used to construct the free commutative Rota-Baxter algebras on $A$~\mcite{GK00}. It was later shown that the mixable shuffle product coincides with the quasi-shuffle product in tensor notations~\mcite{EG06,G12}. We recall the construction below in order to relate to the quasi-shuffle algebras discussed earlier.

In tensor notations, the quasi-shuffle algebra generated by $A$ is the \bfk-module
\begin{equation}
\sha^+_{\bfk}(A):=\bigoplus \limits_{k \geq 0} A^{\ot k}
\mlabel{eq:shapr}
\end{equation}
with the convention $A^{\ot 0}=\bfk$. Equip it with the quasi-shuffle product $\ast$ as in Eq.~\meqref{eq:qshuffle}.

Let $Y=\{y_1, \ldots, y_m \}$ be a set of symbols and let $A=\bfk [Y]$ be the free commutative algebra on $Y$. The quasi-shuffle algebra $\sha^{+}_{\bfk} (\bfk [Y])$ is just another version of the quasi-shuffle subalgebra ${\bf k} W\snx$ of ${\bf k} \snx$ recalled in Proposition~\mref{prop:calg} with $\E=\wmn$ and hence $\sha^{+}_{\bfk} (\bfk [Y])$ is isomorphic to $\wmqsym$ by Theorem~\mref{thm:isomorphism}. For later use, we describe it in more detail.

By Eq.~(\mref{eq:formal}), each map $f: \A \rightarrow \NN$ can be identified with the formal product
\begin{align*}
w(f) := \prod \limits_{\ax \in \mathrm{supp}(f)} y_{\ax}^{f(\ax)} =y_{\ax_{1}}^{\e_1} \cdots y_{\ax_{k}}^{\e_k}.
\end{align*}
As $\A$ is finite, we can write $w(f)$ by
\begin{align}
y_{1}^{\e_1} y_{2}^{\e_2} \cdots y_{m}^{\e_m}
\mlabel{eq:special}
\end{align}
and allow the exponents $\e_i$ to be zero. By abuse of notation, denote by $\wnx$ all products written as Eq.~(\mref{eq:special}):
$$ \wnx:=\{y_1^{a_1}\cdots y_m^{a_m}\,|\, a_i\geq 0, 1\leq i\leq m\}.$$ Note that $\wnx$ is a \bfk-basis of the polynomial algebra $\bfk [Y]$, that is,
\begin{align}
\bfk [Y]=\bfk \{w \mid w \in  \wnx \} .
\mlabel{eq:abasis}
\end{align}

\begin{defn}
A {\bf weak multi-composition} is a sequence $(w_1, \ldots, w_k)$, where each $w_i$ belongs to $\wnx$.
\end{defn}

The empty word $\etree$ is regarded as both a trivial element in $\wnx$ and a trivial weak multi-composition.
Denote the set of all weak multi-compositions by $\wsnx$.
For ${\bf w}=(w_1,  \ldots, w_k) \in \wsnx$, denote
\begin{align}
{\bf w}^{\ot} := w_1  \ot  \cdots \ot w_k.
\mlabel{eq:deftensor}
\end{align}
Here we use the convention that ${\bf w}^{\ot}=1_\bfk$ if ${\bf w}=\etree$. Then by Eq.~\meqref{eq:shapr},
\begin{align}
\sha^+_{\bfk} (\bfk [Y])=\bigoplus \limits_{k \geq 0} \bfk[Y]^{\ot k}=\bigoplus \limits_{{\bf w} \in \wsnx } \bfk {\bf w}^{\ot}.
\mlabel{eq:sha+}
\end{align}

A map $\theta: \wmp  \rightarrow \mathbb{N}$ was given in~\mcite{YGT19} by
\begin{equation*}
\theta: \wmp \rightarrow \mathbb{N}, \quad n \mapsto
\left\{
\begin{array}{ll}
0, & \text{if $n=\varepsilon$}, \\
n, & \text{otherwise},
\end{array}
\right .
\end{equation*}
replacing $\varepsilon$ by 0. Then $\theta$ induces a bijection
\begin{equation}
\theta: \sm^{\wmp} \rightarrow \wnx, \quad w \mapsto
\left\{
\begin{array}{ll}
\etree, &  \text{if $w=1^{\varepsilon}\cdots m^{\varepsilon}$}, \\
y_1^{\theta(n_1)}\cdots y_m^{\theta(n_m)}, & \text{otherwise} ,
\end{array}
\right .
\mlabel{eq:map2}
\end{equation}
and then a bijection
\begin{equation}
\theta: W\snx \rightarrow \wsnx, \quad {\bf w} \mapsto
\left\{
\begin{array}{ll}
\etree, &  \text{if ${\bf w}=\etree$}, \\
(\theta(w_1), \theta(w_2), \cdots, \theta(w_m)), & \text{if ${\bf w}=(w_1, \cdots, w_m)$} \neq \etree.
\end{array}
\right .
\mlabel{eq:map}
\end{equation}

\begin{exam}
For $m=2$ and
$$w=1^{\varepsilon} 2 \in \sm^{\wmp} \,\text{ and }\, {\bf w}=(1 2^{\varepsilon}, 1^{\varepsilon} 2^{\varepsilon}, 1^2 2)\in W\snx,$$ we have
\begin{align*}
\theta(w)=y_{1}^{\theta(\varepsilon)} y_{2}^{\theta(1)}=y_{1}^0 y_{2} = y_{2} \in \wnx
\end{align*}
and
\begin{align*}
\theta({\bf w})=(y_{1}^{\theta(1)} y_{2}^{\theta(\varepsilon)}, y_{1}^{\theta(\varepsilon)} y_{2}^{\theta(\varepsilon)}, y_{1}^{\theta(2)} y_{2}^{\theta(1)})=(y_{1} y_{2}^0, y_{1}^0 y_{2}^0, y_{1}^2 y_{2}) = (y_{1}, \etree, y_{1}^2 y_{2}) \in \wsnx.
\end{align*}
\end{exam}

As $\theta$ induces a bijection between $W\snx$ and $\wsnx$, by Eqs.~(\mref{eq:qshuffle}) and (\mref{eq:sha+}), there is an isomorphism of quasi-shuffle algebras
\begin{align*}
\bfk W\snx \cong \sha^{+}_{\bfk}(\bfk [Y ]).
\end{align*}

Since $\{ M_{\bf w} \mid {\bf w} \in W\snx \}$ is a $\bfk$-basis of  $\wmqsym$ and $\wsnx$ is a $\bfk$-basis of $\sha^+_{\bfk}(\bfk[Y])$ by Eq.~(\mref{eq:sha+}), by Theorem~\mref{thm:isomorphism}, we have

\begin{lemma}
There is an algebra isomorphism $f$ from $\wmqsym$ to the quasi-shuffle algebra $(\sha^+_{\bfk}(\bfk[Y]), \ast)$, where $f(M_{\bf w})=\theta({\bf w})^{\ot}$  with ${\bf w} \in W\snx$.
\mlabel{lem:iso11}
\end{lemma}

\begin{proof}
Since $\bfk W\snx \cong \sha^{+}_{\bfk}(\bfk [Y ])$ as quasi-shuffle algebras, by Theorem~\mref{thm:isomorphism}, $f$ is an algebra isomorphism.
\end{proof}

\subsection{Free Rota-Baxter algebras}
\mlabel{ss:frb}
We begin with some basic concepts of Rota-Baxter algebras~\mcite{G12}.

\begin{defn}
\begin{enumerate}
\item For a fixed $\lambda \in \bfk$, a {\bf Rota-Baxter k-algebra of weight $\lambda$} is a \bfk-algebra $\R$ together with a linear operator $P:\R \rightarrow \R$ that satisfies the Rota-Baxter identity
\begin{align*}
P(x)P(y) = P(xP(y)) + P(P(x)y) + \lambda P(xy) \,\, \text{ for $x, y \in \R.$}
\end{align*}
Such an operator $P$ is called a {\bf Rota-Baxter operator} of weight $\lambda$.
If further $\R$ is commutative, then $(\R,P)$ is called a {\bf commutative Rota-Baxter algebra}.

\item Given two Rota-Baxter algebras $(\R,P)$ and $(\R',P')$ of weight $\lambda$, a map $f: \R \rightarrow \R'$ is called a {\bf morphism of Rota-Baxter algebras} if $f$ is a \bfk-algebra homomorphism and $f \circ P=P' \circ f$.
\end{enumerate}
\end{defn}

The following is the concept of free commutative Rota-Baxter algebras.

\begin{defn}
Let $A$  be a commutative \bfk-algebra. The {\bf free commutative Rota-Baxter \bfk-algebra of weight $\lambda$} on $A$ is a commutative Rota-Baxter \bfk-algebra $(F(A), P_{A})$, together with a \bfk-algebra homomorphism $j_{A}:A \rightarrow F(A)$, with the property that, for any commutative Rota-Baxter \bfk-algebra $(\R, P)$ of weight $\lambda$ and any \bfk-algebra homomorphism $f:A \rightarrow \R$, there is a unique morphism of Rota-Baxter algebras $$\overline{f}:(F(A), P_{A}) \rightarrow (\R, P)$$ such that $f=\overline{f} \circ j_{A}$.
\end{defn}

We are going to recall the construction of free commutative Rota-Baxter \bfk-algebras~\mcite{GK00}.
Define
\begin{align*}
\sha_{\bfk}(A) := A \ot \sha^+_{\bfk}(A)=\bigoplus \limits_{k \geq 1} A^{\ot k}
\end{align*}
to be the tensor product algebra of the algebras $A$ and $\sha_{\bfk}^+(A)$. Let $\diamond$ denote the multiplication on the tensor product algebra.
Also define a linear operator
\begin{equation}
P_{A}:\sha_{\bfk}(A) \rightarrow \sha_{\bfk}(A), \quad \mathbf{a} \mapsto 1_A \ot \mathbf{a}.
\mlabel{eq:defnp}
\end{equation}

\begin{lemma}\mcite{GK00}
Let $A$ be a commutative \bfk-algebra. The triple $(\sha_{\bfk}(A), \diamond, P_{A})$, together with
the natural embedding $A \hookrightarrow \sha_{\bfk}(A)$, is the free commutative Rota-Baxter algebra of weight $1_{\bfk}$ on $A$.
\mlabel{lem:freeRB}
\end{lemma}

Taking $A=\bfk [Y]$ to be the free commutative algebra on a set $Y$, we get

\begin{lemma}\mcite{GK00}
Let $Y=\{y_1, \ldots, y_m \}$ be a set and
\begin{align}
\sha_{\bfk}(Y) :=\sha_{\bfk}(\bfk [Y])=\bfk [Y] \ot \sha_{\bfk}^+(\bfk [Y]).
\mlabel{eq:defsha}
\end{align}
Then the triple $(\sha_{\bfk}(Y), \diamond, P_{\bfk [Y]})$, together with
the natural embedding $Y \hookrightarrow \bfk[Y] \hookrightarrow \sha_{\bfk}(Y)$, is the free commutative Rota-Baxter algebra of weight $1_{\bfk}$ on $Y$.
\mlabel{lem:algebra}
\end{lemma}

By Eqs.~(\mref{eq:sha+}) and (\mref{eq:defsha}),
\begin{align}
\sha_{\bfk}(Y)=\sha_{\bfk}(\bfk[Y]) =\bigoplus \limits_{k \geq 1} \bfk[Y]^{\ot k} = \bigoplus \limits_{{\bf w} \in \wsnx  \setminus \{ \etree \} } \bfk {\bf w}^{\ot}
\mlabel{eq:sha}
\end{align}
with the notation ${\bf w}^\ot$ in Eq.~\meqref{eq:deftensor}.

We next show that a scalar extension of $\wmqsym$ is isomorphic to the free Rota-Baxter algebra on a finite set $Y$.

Let $\X$ be the set of variables
$$	\X := {\bf x}_{\sm}:=\bigsqcup_{i \in \sm} {\bf x}_{i}=\bigsqcup_{i \in \sm} \, \big\{\x_{i,1}, \x_{i,2}, \ldots \big\}.
$$
as given in Eq.~\meqref{eq:eks}. Define two new sets of variables
\begin{align*}
\X_{0}:=\{\x_{1,0}, \ldots, \x_{m,0} \}\,\,  \text{and}\,\, \overline{\X}:= \X_{0} \sqcup \X.
\end{align*}
For $w=1^{n_1} \cdots m^{n_m} \in \sm^{\wmp}$, define as in Eq.~\meqref{eq:ekswj}:
$$x_{w,0}:=\x_{1,0}^{n_1} \cdots \x_{m,0}^{n_m}.$$
In particular, for $\sm^{\varepsilon}=1^{\varepsilon} \cdots m^{\varepsilon}$, we have $$\x_{\sm^{\varepsilon},0}=\x_{1,0}^{\varepsilon} \cdots \x_{m,0}^{\varepsilon}.$$
Then
\begin{align}
\x_{\sm^{\varepsilon},0} \bfk [\X_0]=\bfk \left\{ \x_{1,0}^{n_1 +\varepsilon} \cdots \x_{m,0}^{n_m+\varepsilon} \mid n_i \in \mathbb{N} \, \text{for $1 \leq i \leq m$}  \right\}=\bfk \left\{\x_{w,0} \mid w \in \A^{\wmp} \right\}.
\mlabel{eq:1kx}
\end{align}
Here each $n_i +\varepsilon\neq 0$, $w = 1^{n_1+\varepsilon}\cdots m^{n_m+\varepsilon}\in \A^{\wmp}$ and $x_{w,0}=\x_{1,0}^{n_1 +\varepsilon} \cdots \x_{m,0}^{n_m+\varepsilon}$.

Since $\x_{\sm^{\varepsilon},0}$ is idempotent, $\x_{\sm^{\varepsilon},0} \bfk [\X_0]$ is an algebra with identity $\x_{\sm^{\varepsilon},0}$.
Then we define a subalgebra $\sqsymnx$ of $\bfk [[\overline{\X}]]^{\wmn}$ by
\begin{align}
\sqsymnx :=\x_{\sm^{\varepsilon},0} \bfk [\X_0] \wmqsym \cong \x_{\sm^{\varepsilon},0} \bfk [\X_0] \ot \wmqsym.
\mlabel{eq:constru}
\end{align}

The multiplication of $\sqsymnx$ can be made explicitly. By Eq.~(\mref{eq:1kx}),
\begin{equation}
\begin{aligned}
\sqsymnx=&\ \bfk \{x_{w,0}M_{{\bf w}} \mid w \in \A^{\wmp}, {\bf w} \in W\snx  \}\\
=&\ \bigoplus \limits_{(w, {\bf w}) \in W\snx \setminus \{\etree \}} \bfk \overline{M}_{(w,{\bf w})},
\mlabel{eq:sqsym}
\end{aligned}
\end{equation}
where $\overline{M}_{(w,{\bf w})}:=x_{w,0}M_{\bf w}$.
In terms of the quasi-shuffle product $\ast$ in Eq.~(\mref{eq:qshuffle}), we have
\begin{equation*}
\begin{aligned}
\overline{M}_{(u_0,{\bf w}_1)} \overline{M}_{(v_0,{\bf w}_2)}=x_{u_0,0} M_{{\bf w}_1} x_{v_0,0} M_{{\bf w}_2}= x_{u_0,0} x_{v_0,0} M_{{\bf w}_1} M_{{\bf w}_2}=\ x_{u_0 \cdot v_0,0} M_{{\bf w}_1 \ast {\bf w}_2}
=\ \overline{M}_{(u_0 \cdot v_0, {\bf w}_1 \ast {\bf w}_2)},
\label{eq:augshuffle1}
\end{aligned}
\end{equation*}
for $u_0,v_0 \in \A^{\wmp} $ and ${\bf w}_1,{\bf w}_2 \in W\snx$. Notice that $\overline{M}_{(\sm^{\varepsilon})}$ is the identity element since
$$\overline{M}_{(\sm^{\varepsilon})} \overline{M}_{(w,{\bf w})}=\overline{M}_{(\sm^{\varepsilon} \cdot w, \etree \ast {\bf w})}=\overline{M}_{(w,{\bf w})}.$$

Next we define a linear map:
\begin{align}
P: \sqsymnx \rightarrow \sqsymnx, \quad \overline{M}_{(w,\, {\bf w})} \mapsto   \overline{M}_{(\sm^{\varepsilon},\,w,\, {\bf w})}.
\mlabel{eq:operator}
\end{align}

\begin{exam}
For $m=2$ and ${\bf w}=(w)$ with $w=1 2^{\varepsilon} \in \A^{\wmp}$, we have
\begin{align*}
\overline{M}_{{\bf w}}=&\ \x_{w,0}M_{\etree}=\x_{1,0} \x_{2,0}^{\varepsilon} M_{\etree}=\x_{1,0} \x_{2,0}^{\varepsilon},
\end{align*}
and
\begin{align*}
P(\overline{M}_{{\bf w}})=&\ \overline{M}_{(\sm^{\varepsilon}, {\bf w})}=\x_{\sm^{\varepsilon},0}M_{{\bf w}}=\x_{1,0}^{\varepsilon} \x_{2,0}^{\varepsilon} \sum\limits_{1 \leq j} \x_{1,j} \x_{2,j}^{\varepsilon}.
\end{align*}
For $ {\bf w}=(w_0,w_1)$ with $w_0=1^{\varepsilon}2$ and $w_1= 1 2^{\varepsilon} \in \A^{\wmp}$, we have
\begin{align*}
\overline{M}_{{\bf w}}=&\ \x_{w_0,0}M_{(w_1)}=\x_{1,0}^{\varepsilon} \x_{2,0} M_{(1^1 2^{\varepsilon})}=\x_{1,0}^{\varepsilon} \x_{2,0} \sum\limits_{1 \leq j} \x_{1,j} \x_{2,j}^{\varepsilon}
\end{align*}
and by Eq.~(\mref{eq:mi}),
\begin{align*}
P(\overline{M}_{{\bf w}})=&\ \overline{M}_{(\sm^{\varepsilon},{\bf w})}=\x_{\sm^{\varepsilon},0}M_{{\bf w}}=\x_{1,0}^{\varepsilon} \x_{2,0}^{\varepsilon}\sum\limits_{1 \leq j_1< j_2} \x_{1,j_1}^{\varepsilon} \x_{2,j_1} \x_{1,j_2} \x_{2,j_2}^{\varepsilon}.
\end{align*}
\end{exam}

Now we are ready for our result confirming Rota's proposal applying Rota-Baxter algebras to give generalizations of symmetric functions.

\begin{theorem}
For a fixed positive integer $m$, the pair $(\sqsymnx,P)$, together with the injection
$$Y \hookrightarrow \sqsymnx, \quad y_i \mapsto x_{1,0}^{\varepsilon} \cdots x_{{i-1},0}^{\varepsilon}\, x_{i,0} \, x_{{i+1},0}^{\varepsilon}\cdots x_{m,0}^{\varepsilon},$$
is the free commutative Rota-Baxter algebra of weight $1_\bfk$ on $Y$.
\mlabel{thm:free}
\end{theorem}

\begin{proof}
On the one hand, it follows from Eq.~(\mref{eq:sha}) that $\sha_{\bfk}(Y)$ has a \bfk-basis
\begin{align*}
\{w_0 \ot w_1 \ot \cdots \ot w_n \mid (w_0,w_1, \ldots, w_n) \in \wsnx  \setminus \{\etree \} \}.
\end{align*}
On the other hand, by Eq.~(\mref{eq:sqsym}), $\sqsymnx$ has a \bfk-basis
\begin{align*}
&\ \{\overline{M}_{(w_0,{\bf w})} \mid w_0 \in \A^{\wmp}, {\bf w}=(w_1,\ldots, w_n) \in W\snx  \}\\
=&\ \{\overline{M}_{(w_0, w_1, \cdots, w_n)} \mid (w_0, w_1, \ldots, w_n) \in W\snx \setminus \{ \etree\} \}.
\end{align*}
From Eq.~(\mref{eq:map}), there is a linear isomorphism
\begin{align}
f: \sqsymnx \rightarrow \sha_{\bfk}(Y), \quad \overline{M}_{(w_0, w_1, \ldots, w_n)}\mapsto \theta(w_0) \ot \theta(w_1) \ot \cdots \ot \theta(w_n).
\mlabel{eq:deff}
\end{align}
Then $f$ is an algebra isomorphism by taking the tensor product of the algebra isomorphisms $ \x_{\sm^{\varepsilon},0} \bfk [\X_0]\cong \bfk[Y]$ and $W\mqsymn \cong \sha^+_{\bfk}(\bfk[Y])$ from Lemma~\mref{lem:iso11}.

Further
\begin{align*}
&\ f \circ P(\overline{M}_{(w_0, w_1, \ldots, w_n)}) =f(\overline{M}_{(\sm^{\varepsilon},w_0, w_1, \ldots, w_n)}) \quad \quad (\text{by Eq.~(\mref{eq:operator})})\\
=&\ \theta(\sm^{\varepsilon}) \ot \theta(w_0) \ot \cdots \ot \theta(w_n) \quad \quad (\text{by Eq.~(\mref{eq:deff})})\\
=&\ 1_{\bfk [Y]} \ot \theta(w_0) \ot \cdots \ot \theta(w_n) \quad \quad (\text{by Eq.~(\mref{eq:map2})})\\
=&\ P_{\bfk [Y]}(\theta(w_0) \ot \cdots \ot \theta(w_n)) \quad \quad (\text{by Eq.~(\mref{eq:defnp})})\\
=&\ P_{\bfk [Y]} \circ f(\overline{M}_{(w_0, w_1, \ldots, w_n)}) \quad \quad (\text{by Eq.~(\mref{eq:deff})}),
\end{align*}
and so $f$ is a Rota-Baxter algebra isomorphism. Hence $(\sqsymnx,P)$ is the free commutative Rota-Baxter algebra of weight $1_\bfk$ on $Y$ by Lemma~\mref{lem:algebra}.
\end{proof}

Finally, we give a Hopf algebra structure on $\sqsymnx$.
By Eq.~(\mref{eq:1kx}), the algebra $\x_{\sm^{\varepsilon},0} \bfk [\X_0]$ has a \bfk-basis
$$\{x_{w,0} \mid w \in \A^{\wmp } \}.$$
From Eq.~({\mref{eq:abasis}}), $\bfk [Y]$ has a \bfk-basis
$$\{ w \mid  w\in \wnx \}.$$
In terms of the bijection given in Eq.~(\mref{eq:map2}), there is a linear isomorphism
\begin{align*}
f: \x_{\sm^{\varepsilon},0} \bfk [\X_0] \rightarrow \bfk [Y],\quad  x_{w,0} \mapsto \theta(w).
\end{align*}
Moreover, for $u,v \in \A^{\wmp}$,
\begin{align*}
&\ f(\x_{u,0})f(\x_{v,0})=\theta(u) \theta(v)=\theta(u \cdot v) \quad \quad (\text{by the commutativity of $\bfk [Y]$})\\
=&\ f(\x_{u \cdot v,0})=f(\x_{u,0} \x_{v,0}),
\end{align*}
hence $f$ is an algebra isomorphism. Since the polynomial algebra $\bfk [Y]$ is a Hopf algebra, the algebraic isomorphism $f$ induces a Hopf algebraic structure on $\x_{\sm^{\varepsilon},0} \bfk [\X_0]$.

Recall that if  $H_1$ and $H_2$ are two Hopf algebras, then $H_1\ot H_2$ is also a Hopf algebra~\mcite{A80}.

\begin{prop}
Let $Y=\{y_1,\ldots,y_m\}$ be a set. The free commutative Rota-Baxter algebra $(\sqsymnx,P)$ of weight $1_\bfk$ on $Y$ is a Hopf algebra.
\end{prop}

\begin{proof}
By Eq.~(\mref{eq:constru}), there is an algebra isomorphism
\begin{align*}
\sqsymnx \cong \x_{\sm^{\varepsilon},0} \bfk [\X_0] \ot \wmqsym.
\end{align*}
Notice that $\wmqsym$ is a Hopf algebra by Corollary~\mref{cor:subhopf}
and $\x_{\sm^{\varepsilon},0} \bfk [\X_0]$ is a Hopf algebra by the above discussion.
Hence $\sqsymnx$ is a Hopf algebra.
\end{proof}

\noindent {\bf Acknowledgments.}
This work is supported by the National Natural Science Foundation of China (Grant No. 12071191,12301025) and the Innovative Fundamental Research Group Project of Gansu Province (Grant No. 23JRRA684).

\smallskip

\noindent
{\bf Declaration of interests. } The authors have no conflicts of interest to disclose.

\smallskip

\noindent
{\bf Data availability. } No new data were created or analyzed in this study.

\end{document}